\documentclass[reqno,english]{amsart}

\usepackage{amsmath,amsfonts,amssymb,graphicx,amsthm,enumerate,url}
\usepackage[noadjust]{cite}
\usepackage{stmaryrd}
\usepackage{comment,paralist,mathrsfs}
\usepackage{mathrsfs,booktabs,tabularx}
\usepackage{xifthen,xcolor,tikz,setspace}
\usetikzlibrary{decorations.pathmorphing,patterns,shapes,calc,decorations}
\usetikzlibrary{decorations.pathreplacing}
\usepackage{mathtools}
\usepackage[small]{caption}
\usepackage{subcaption}

\usepackage[colorinlistoftodos]{todonotes}
\usepackage[colorlinks=true]{hyperref}
\numberwithin{equation}{section}

\renewcommand{\epsilon}{\varepsilon}

\newcommand{\one}{\mathbf{1}}
\usepackage{color}

\newtheorem{maintheorem}{Theorem}

\newtheorem{theorem}{Theorem}[section]
\newtheorem*{theorem*}{Theorem}
\newtheorem{lemma}[theorem]{Lemma}

\newtheorem{claim}[theorem]{Claim}
\newtheorem{proposition}[theorem]{Proposition}

\newtheorem{observation}[theorem]{Observation}
\newtheorem*{observation*}{Observation}
\newtheorem{fact}[theorem]{Fact}
\newtheorem{corollary}[theorem]{Corollary}

\newtheorem{remark}[theorem]{Remark}

\theoremstyle{definition}{

\newtheorem{definition}[theorem]{Definition}
\newtheorem*{definition*}{Definition}

\newtheorem{question}[theorem]{Question}
}

\newcommand{\E}{\mathbb E}

\renewcommand{\P}{\mathbb P}

\newcommand{\R}{\mathbb R}

\newcommand{\cD}{\mathcal D}
\newcommand{\cG}{\mathcal G}
\newcommand{\cE}{\mathcal E}
\newcommand{\cA}{\mathcal A}

\newcommand{\cL}{\mathcal L}
\newcommand{\cP}{\mathcal P}

\newcommand{\cY}{\mathcal Y}
\newcommand{\fa}{\mathfrak a}
\newcommand{\fb}{\mathfrak b}
\newcommand{\ff}{\mathfrak f}

\newcommand{\sC}{\mathsf C}
\newcommand{\bo}{\mathbf o}
\newcommand{\bu}{\mathbf u}
\newcommand{\bv}{\mathbf v}
\newcommand{\bw}{\mathbf w}
\newcommand{\sL}{\mathscr L}
\newcommand{\aG}{\fa}
\newcommand{\aGhat}{\widehat\fa}
\newcommand{\bG}{\fb}
\newcommand{\leqP}{\mathbin{\leq_{\textsc p}}}
\newcommand{\nleqP}{\mathbin{\nleq_{\textsc p}}}
\DeclareMathOperator{\var}{Var}

\DeclareMathOperator{\rank}{rank}
\begin{document}

\title{The threshold for stacked triangulations}

\author{Eyal Lubetzky}
\address{E.\ Lubetzky\hfill\break
Courant Institute\\ New York University\\
251 Mercer Street\\ New York, NY 10012,~USA.}
\email{eyal@courant.nyu.edu}

\author{Yuval Peled}
\address{Y.\ Peled\hfill\break
Einstein Institute of Mathematics\\
 Hebrew University\\ Jerusalem~91904\\ Israel.}
\email{yuval.peled@mail.huji.ac.il}

\begin{abstract}
A \emph{stacked triangulation} of a $d$-simplex $\bo=\{1,\ldots,d+1\}$ ($d\geq 2$) is a triangulation obtained by repeatedly subdividing a $d$-simplex into $d+1$ new ones via a new vertex (the case $d=2$ is known as an Appolonian network). We~study the occurrence of  such a triangulation in the Linial--Meshulam model, i.e., for which $p$ does the random simplicial complex $Y\sim \cY_d(n,p)$ contain the faces of a stacked triangulation of the $d$-simplex $\bo$, with its internal vertices labeled in $[n]$.  In the language of bootstrap percolation in hypergraphs, it pertains to the threshold for $K_{d+2}^{d+1}$, the $(d+1)$-uniform clique  on $d+2$ vertices.

Our main result identifies this threshold for every $d\geq 2$, showing it is asymptotically $(\alpha_d n)^{-1/d}$, where $\alpha_d$ is the growth rate of the Fuss--Catalan numbers of order $d$. The proof hinges on a second moment argument in the supercritical regime, and on Kalai's algebraic shifting in the subcritical regime.
\end{abstract}

{\mbox{}
\vspace{-.4cm}
\maketitle
}
\vspace{-0.4cm}

\section{Introduction}

We study the following question, which can be phrased, and different flavors of it were studied, in the context of (i) stacked triangulations of a simplex in a random simplicial complex, or (ii) the homotopy-equivalent closure of a  simplicial complex under elementary expansions, or (iii) bootstrap percolation on random hypergraphs.
For simplicity, we first phrase it for dimension $d=2$, noting  the objects of interest, \emph{stacked triangulations}, are in that case also known as \emph{Appolonian networks}.

\begin{question}
\label{q:stacked}
A \emph{stacked triangulation} of a triangle (a.k.a.\ \emph{Appolonian network}) is obtained by repeatedly subdividing a triangle into three via a new~vertex. For a (random) set of triangles $Y \subseteq \binom{[n]}3$, is there such a triangulation of $\{1,2,3\}$ and a labeling of its vertices in $[n]$ so that all these labeled faces belong to~$Y$? 
\end{question}
If $Y$ contains each triangle in $\binom{[n]}3$ independently with probability $p$, this addresses such  triangulations in the Linial--Meshulam~\cite{LM06} random complex $\cY_2(n,p)$. See~\S\ref{subsection:complexes} for  related work on triangulations and simple connectivity of $\cY_2(n,p)$.

The equivalent formulation obtained from reversing this process is as follows.
{
\setcounter{theorem}{0}
\renewcommand\thetheorem{\arabic{section}.\arabic{theorem}'}
\begin{question}\label{q:bootstrap}
Given a (random) set of triangles $Y\subseteq \binom{[n]}3$, repeatedly add to~$Y$ every triangle $f=\{v_1,v_2,v_3\}$ for which there exists some $u\in[n]$ such that all three triangles $\{v_1,v_2,u\}$, $\{v_1,u,v_3\}$, $\{u,v_2,v_3\}$ belong to $Y$. 
When this process terminates, does the final set $Y_\infty$ contain $\{1,2,3\}$? When does one have $Y_\infty = \binom{[n]}3$?
\end{question}
}

In the language of weak saturation and bootstrap percolation of hypergraphs introduced by Bollob\'as~\cite{Bol68} in the late 1960's, this is $K^3_4$-bootstrap percolation: one starts with a random $3$-uniform hypergraph, and
 repeatedly adds a hyperedge if it completes a new copy of $K_{4}^3$, the complete hypergraph on $4$ vertices (see \S\ref{subsection:bootstrap}).

Another way of viewing Question~\ref{q:bootstrap} is in terms of Whitehead's \emph{elementary moves} between simplicial complexes~\cite{Wh39}.  A \emph{$3$-expansion} in a simplicial complex is the addition of a $3$-simplex and the single $2$-face missing from it---a fundamental operation that preserves its homotopy type.  In this case, one adds the tetrahedron $\{v_1,v_2,v_3,u\}$ along with the missing face $f$. The above question thus asks  what complex we arrive at by after performing all possible $3$-expanions on $\cY_2(n,p)$.

We now state the setting of the problem formally for any fixed dimension $d\geq 2$.
A bistellar $0$-move on a $d$-dimensional simplex is a subdivision of the simplex into $d+1$ simplicies using a new vertex in its interior. Combinatorially, the $d$-simplex\footnote{As a member of a simplicial complex, a set of size $d+1$ is called a $d$-dimensional face or simplex.} $f=\{v_1,\ldots,v_{d+1}\}$ is replaced by $f \setminus \{v_j\} \cup \{u\}$, $j=1,\ldots,d+1$, where $u$ is the new vertex. A {\em stacked triangulation} is a triangulation of the $d$-dimensional simplex that is obtained from the simplex by a sequence of bistellar $0$-moves. Let $Y$ be a simplicial complex containing the boundary of the $d$-simplex $\{v_1,\ldots,v_{d+1}\}$. We say that $Y$ admits a {\em stacked contraction} of $f$ if there is a simplicial map from a stacked triangulation $T$ to $Y$ that maps the boundary to $f$, or in other words, if there is a labelling of the vertices of $T$ in which the boundary vertices are labeled $v_1,\ldots,v_{d+1}$ and the label-set of every face of $T$ belongs to $Y$. 

    In this paper we study a stacked triangulations and contractions in the binomial Linial--Meshulam random simplicial complex. More precisely, we let $Y\sim \cY_d(n,p)$ be an $n$-vertex $d$-dimensinal complex with a full skeleton in which every $d$-simplex appear independently with probability $p=p(n)$, and ask whether there exists a stacked contraction of the $d$-simplex $\{1,\ldots,d+1\}$ in $Y$ with high probability (a probability tending to $1$ as $n\to\infty$). Moreover, how many $d$-simplices in $\binom{[n]}{d+1}$ have a stacked contraction in $Y$?

For the alternate point-of-view on this problem, let $Y\subseteq\binom{[n]}{d+1}$ be an initial set of $d$-dimensional faces. 
As long as there exist $f=\{v_1,...,v_{d+1}\}\notin Y$ and $u$ such~that
\[
 f\setminus\{v_{i}\}\cup \{u\} \in Y\qquad\mbox{for all $i$}\,,
\]
the face $f$ is added to $Y$. This deterministic process terminates if no such $f$ exists.

In the language of Whitehead's elementary moves, this is phrased as follows. 

An elementary $(d+1)$-{\em collapse} is the removal of a $(d+1)$-dimensional face $\hat f$ and one of its $d$-dimensional facets $f$ from a complex in which $f$ is not contained in any other $(d+1)$-dimensional face. The inverse of this move, in which a $(d+1)$-simplex $\hat f$ and its $d$-facet $f$ are added to a complex that contains all the boundary of $\hat f$ except $f$, is called an elementary $(d+1)$-{\em expansion}. Both these moves preserve the homotopy type of the complex, and simple-homotopy theory studies the conditions under which homotopy equivalent complexes can be obtained one from the other by a sequence of elementary collapses and expansions~\cite{cohen12}.

Elementary collapses play a big role in the study of random simplicial complexes (see \S\ref{subsection:complexes}). In our context, we start with the random simplicial complex $\cY_d(n,p)$, and carry out an elementary $(d+1)$-expansion as long as such a move is possible. The above questions address, e.g., whether $\{1,\ldots,d+1\}$ is in the final complex.

The $H$-bootstrap percolation process is a deterministic process which starts with an initially infected set $Y\subseteq\binom{[n]}{d+1}$, and in every step a new hyperedge that is the only non-infected hyperedge in a copy of some fixed hypergraph $H$ gets infected. In our setting, the initially infected hypergraph is a binomial random hypergraph in which every hyperedge appears independently with probability $p$ and the fixed hypergraph $H$ is $K_{d+2}^{d+1}$ --- the $(d+1)$-uniform clique on $d+2$ vertices. We wish to estimate the probability that a given hyperedge is infected, as well as the probability for {\em percolation}, i.e., the infection of all hyperedges in $\binom{[n]}{d+1}$. Similar problems in graph bootstrap percolation were investigated thoroughly (see~\S\ref{subsection:bootstrap} for details). Here we settle these problems for the cliques $K_{d+2}^{d+1}$, perhaps one of the most basic problems in random hypergraph bootstrap percolation.

\begin{figure}
    \centering
    \vspace{-0.1in}
    \includegraphics[width=0.55\textwidth]{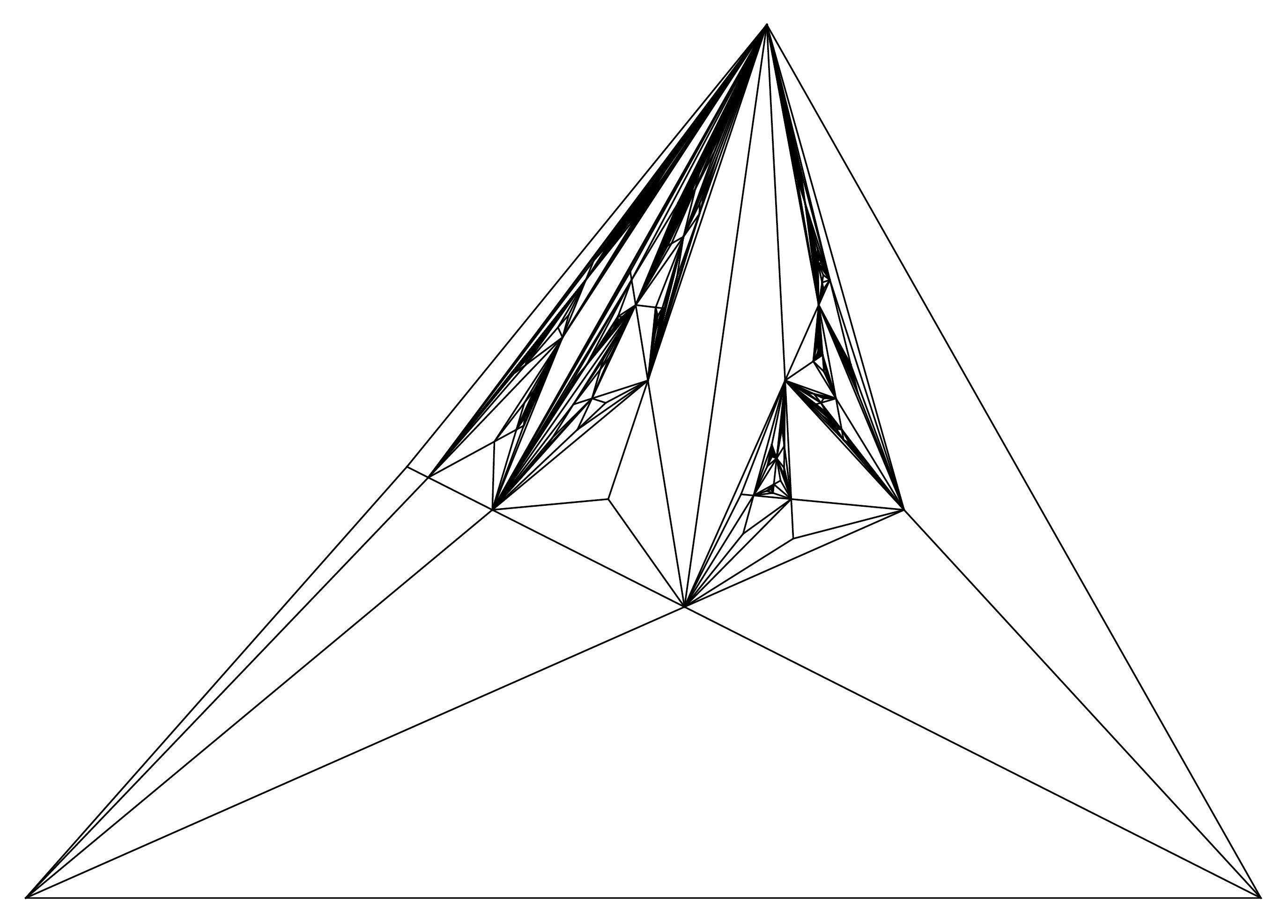}
    \hspace{-0.12\textwidth}\raisebox{0.5in}{\includegraphics[width=0.55\textwidth]{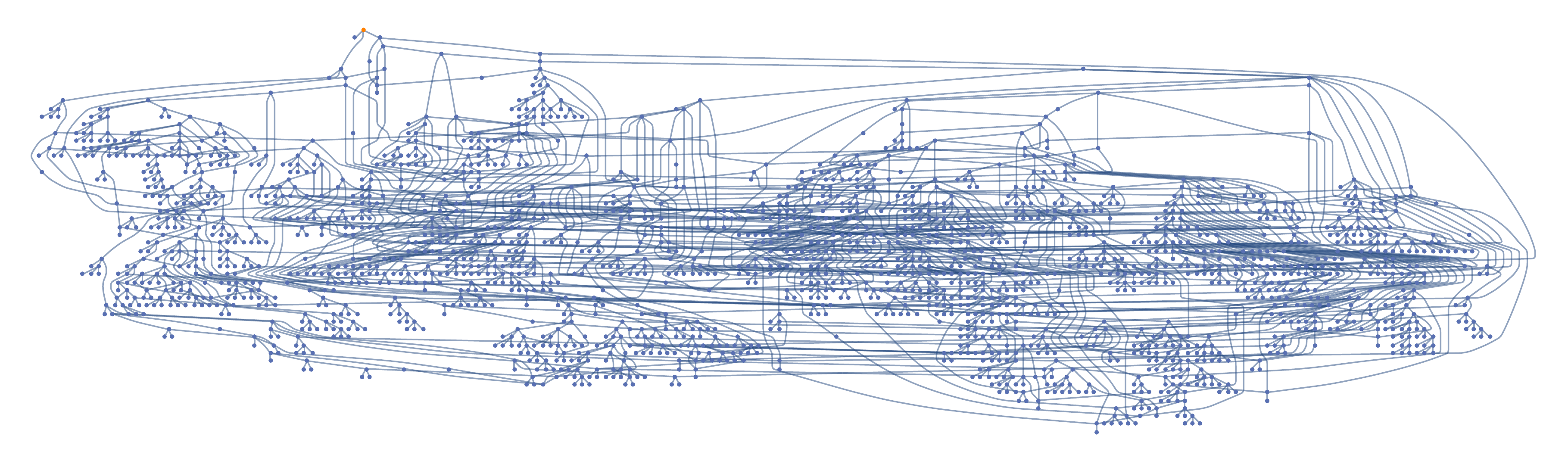}}
    \vspace{-0.1in}
    \caption{A stacked triangulation of $\{1,2,3\}$ in $\cY_2(n,p)$ at criticality for $n=200$, and the corresponding directed acyclic graph~(see Def.~\ref{def:mls}); it has $2097$ vertices, $1275$ of which are leaves, using $184$ labels from $[n]$.}
    \label{fig:tri-n200}
\end{figure}

Our main result determines, for every $d\ge 2$, the asymptotic threshold $p_c$ for the property that $\cY_d(n,p)$ admits stacked contractions of all $f\in\binom{[n]}{d+1}$, or equivalently, the asymptotic threshold for $K_{d+2}^{d+1}$-percolation (see Fig.~\ref{fig:stacked-tri} for a triangulation at~$p_c$).
In what follows, let
\begin{equation}\label{eq:alpha-d}
\alpha_d=(d+1)^{d+1}/d^d
\end{equation}
denote the growth rate of the Fuss--Catalan numbers of order $d$ (see~\S\ref{subsection:stacked}).

\begin{maintheorem}\label{thm:1}
Let $Y\sim\cY_d(n,p)$ for fixed $d\geq 2$, and set $\alpha_d
$ as in~\eqref{eq:alpha-d}.
\begin{enumerate}[(i)]
    \item{} [Subcritical] If $p=\big(1-\frac 1{\log{n}} \big)(\alpha_d n)^{-1/d}$ then the probability that $Y$ admits a stacked contraction of the face $\bo=\{1,\ldots,d+1\}$ is $n^{-1/d+o(1)}$.
    \item{} [Supercritical]
    There exists $\delta_d>0$ such that,
    if $p=\left(1+n^{-\delta_d} \right)(\alpha_d n)^{-1/d}$ then, with high probability, $Y$ admits a stacked contraction of every face $f\in\binom{[n]}{d+1}$.
\end{enumerate}
\end{maintheorem}

For a heuristic argument suggesting that $(\alpha_d n)^{-1/d}$ is the correct threshold, let us completely ignore the delicate dependencies between the stacked contractions of various faces; suppose instead that each face has such a contraction independently with probability $\theta=\theta(n)$, setting aside the trivial equilibrium $\theta=1$. Then a given face $\bo$ has such a contraction either by being in $Y$, which has probability $p$, or via a new vertex, which has probability $1-(1-\theta^{d+1})^{n-(d+1)} \approx n \theta^{d+1}$ (we should have $\theta \ll n^{-1/(d+1)}$ so the final expression could give back $\theta$). This yields the equation
\[ \theta = p + n \theta^{d+1}\,,\]
and we see that if $p = \gamma n^{-1/d}$ and $\theta = \hat\gamma n^{-1/d}$ then this becomes $\hat\gamma^{d+1} - \hat\gamma + \gamma = 0$. The polynomial $Q_\gamma(x)=x^{d+1}-x+\gamma$ is convex in $\R_+$ with $Q_\gamma(x_0) = -\alpha_d^{1/d}$~at its local minimum $x_0=(d+1)^{-1/d}$.
If $\gamma > \alpha_d^{1/d}$ then $Q_\gamma(x) $ has no roots in $\R_+$, leaving us only with the trivial solution $\theta=1$. If $\gamma<\alpha_d^{1/d}$ then $Q_\gamma(x)$ has two roots in $\R_+$, and the solution for $\theta$ should correspond to the smaller root $0<\hat\gamma<x_0$.

The next more detailed result on the subcritical regime confirms this heuristic.
\begin{maintheorem}\label{thm:density}
Consider
$Y\sim\cY_d(n,\gamma n^{-1/d})$ for fixed  $d\geq 2$ and $0\le\gamma<\alpha_d^{-1/d}$.
Let~$\hat X$ be the fraction of sets $\bo\in\binom{[n]}{d+1}$ so that $Y$ admits a stacked contraction of~$\bo$. Then $n^{1/d} \hat X \to \hat\gamma$ in probability as $n\to\infty$, where $\hat\gamma$ is the smallest positive root of the polynomial $Q_\gamma(x) = x^{d+1}-x+\gamma$.
\end{maintheorem}

Theorems~\ref{thm:1} and~\ref{thm:density} show that if $p = \gamma n^{-1/d}$ for $\gamma<\alpha_d^{1/d}$ then the $K_{d+2}^{d+1}$-bootstrap percolation process starting from a binomial random hypergraph with edge density~$p$ does not have a real impact --- the density of the infected hyperedges remains of order $n^{-1/d}$ as it was initially. On the other hand, if $p= \gamma n^{-1/d}$ for $\gamma>\alpha_d^{1/d}$ then the process percolates and with high probability all the hyperedges get infected. In fact, numerical simulations suggest that if we increase the initially infected set one random hyperedge at a time, this cascade occurs at one critical step.

\begin{figure}
    \centering
    \includegraphics[width=0.6\textwidth]{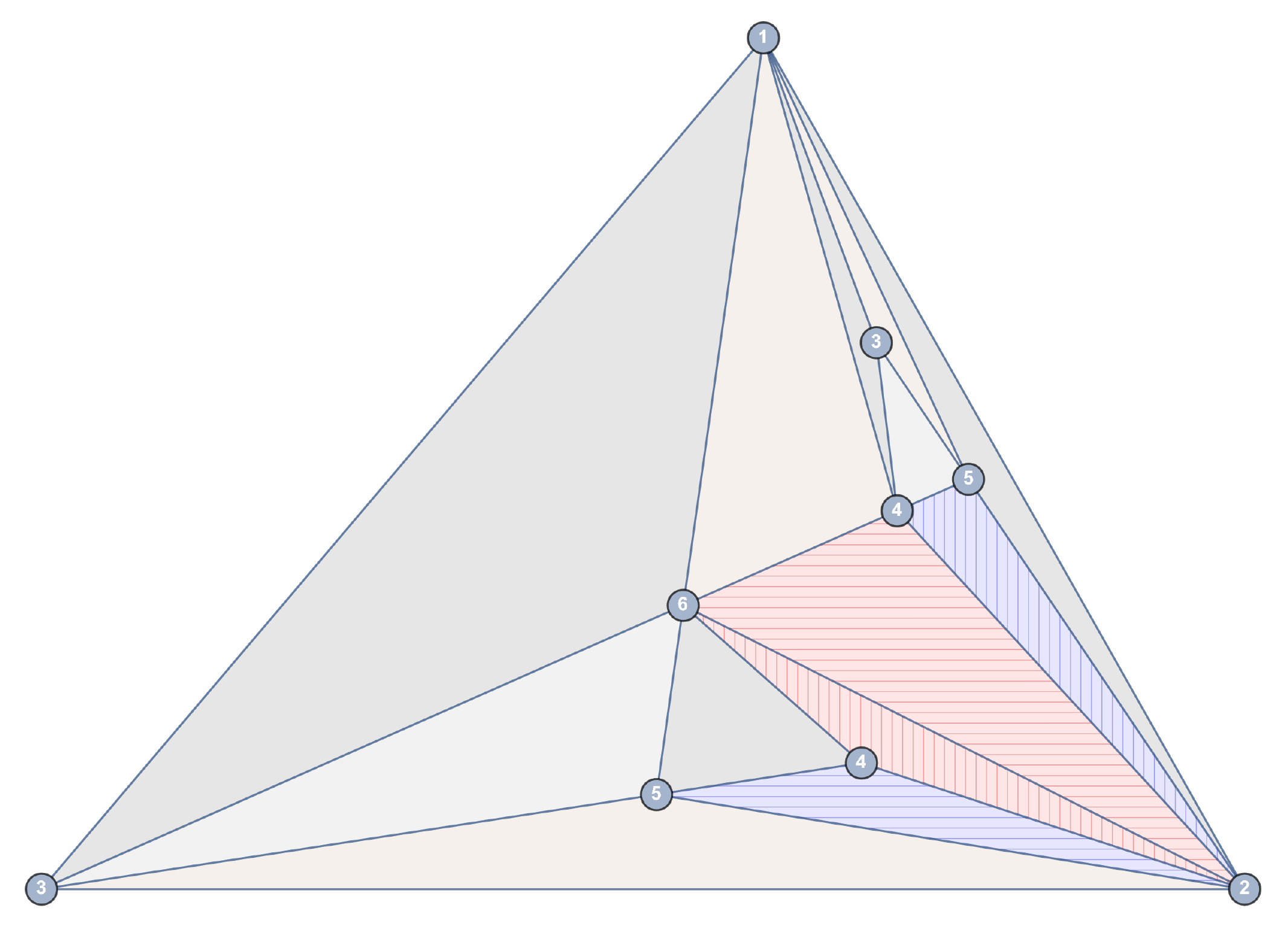}
    \hspace{-0.02\textwidth}
    \raisebox{-7pt}{\includegraphics[width=0.4\textwidth]{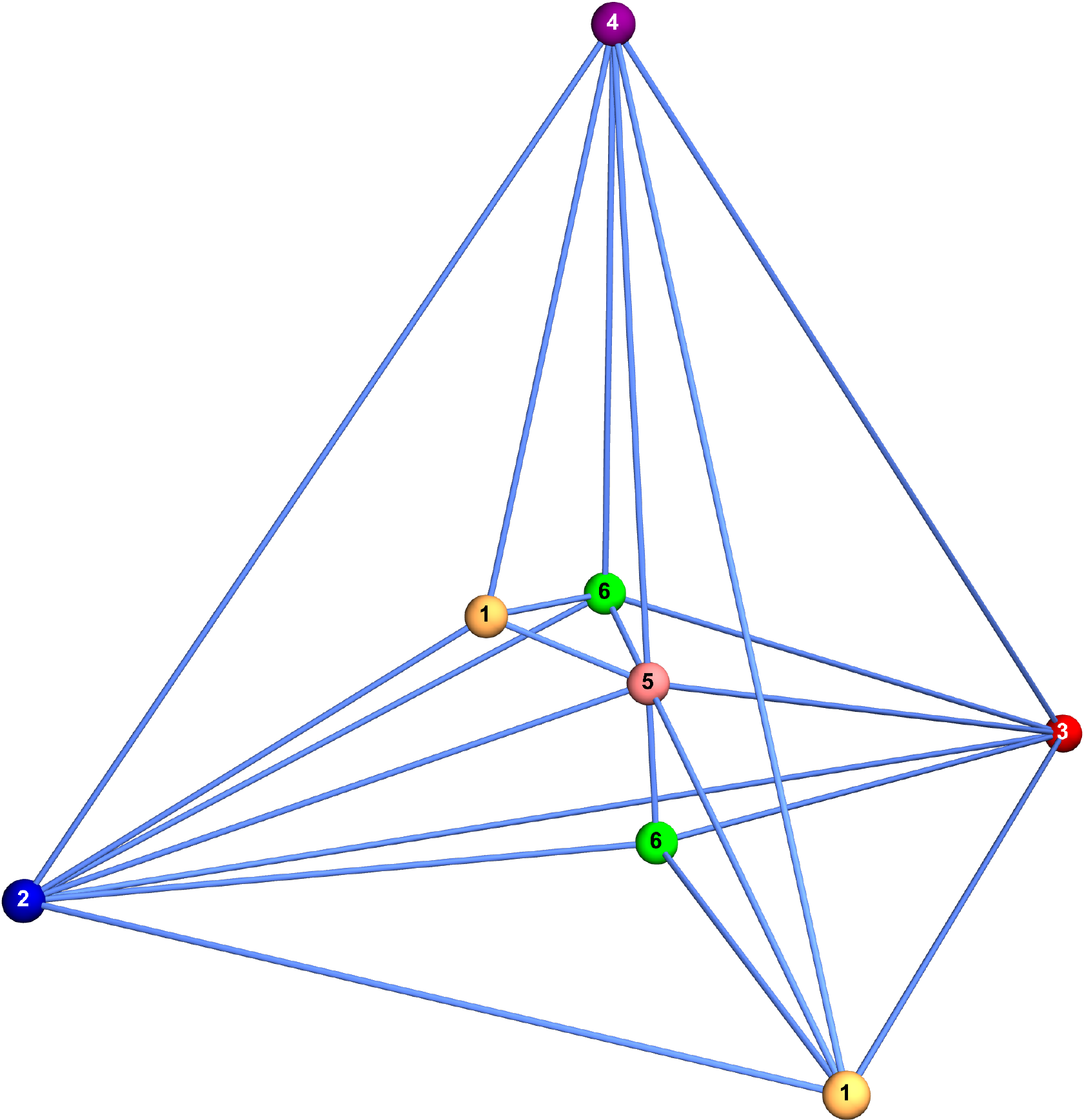}}
    \caption{Stacked contraction in dimension $d=2$ (left; the labels $3,4,5$ and the shaded faces $\{2,4,6\},\{2,4,5\}$ are repeated) and $d=3$ (right; the labels $1,6$ and faces $\{1,2,4,5\},\{1,2,5,6\},\{2,3,5,6\}$ are repeated).}
    \vspace{-0.05in}
    \label{fig:stacked-tri}
\end{figure}

The supercritical regime in Theorem~\ref{thm:1}, where $\cY_d(n,p)$ admits a stacked contraction of every $\bo\in\binom{[n]}{d+1}$, is handled by establishing the stronger property that $\bo$ has a stacked triangulation in which all \emph{labels are distinct}.
The standard way of approaching this would be via the second moment method; however, already for a single face $\bo$, the vanilla method fails: the number of stacked triangulations of~$\bo$ is not concentrated at $p=(1+o(1))p_c$, nor in fact at $p=C p_c$ for any fixed $C>1$, due to  correlations between triangulations. To bypass this, and further show that with high probability there is a stacked contraction of \emph{every} face, we use a modified second moment, restricted
to a set of \emph{balanced} stacked triangulations where each of the $d$-simplices subdividing $\bo$ have the same number of internal vertices.

Conversely, in the subcritical regime we are required to exclude the presence of stacked contractions with {\em repeated labels}, whose combinatorial complexity presents a substantial challenge (see Fig.~\ref{fig:stacked-tri}). This is an inherent difficulty in bootstrap percolation and related problems; e.g., the $2$-dimensional non-stacked variant of this problem --- determining whether a given simplicial complex is simply-connected --- is undecidable. In a random setting, determining the sharp threshold for simple connectivity of $\cY_2(n,p)$, is an important open problem in combinatorial topology. 

In the stacked case, an adaptation of the existing graph bootstrap percolation machinery due to Balogh, Bollob\'as and Morris~\cite{BBM} shows $p_c > n^{-1/d} (\log n)^{-c_d}$, where $c_d>0$ increases with the dimension~$d$. 
While we could decrease $c_d$ in the superfluous poly-log factor by a careful analysis of triangulations with repeated labels, neither this refinement, nor any of the other combinatorial first moment arguments that we attempted, could recover the correct order of the critical probability.

Miraculously, Kalai's algebraic shifting procedure~\cite{kalai2002} enabled us to establish new non-trivial bounds on the number and complexity of stacked contractions (roughly put, one aims to bound the tree excess in the directed acylic graph corresponding to a stacked contraction via the number of labels it features --- e.g., in Fig.~\ref{fig:tri-n200} the former is $370$ and the latter is $184$; see Theorem~\ref{thm:a-b}), leading to the sharp threshold.

\subsection{Related works}\label{sec:related-work}
Next we discuss a number of lines of research that pertain to the problem studied here, and mention related literature on these.

\subsubsection{Random simplicial complexes}
The systematic study of the topology of random complexes was initiated by Linial and Meshulam~\cite{LM06} who introduced $\cY_2(n,p)$ and studied its homological connectivity. The $d$-dimensional generalization of this model $\cY_d(n,p)$ was introduced in~\cite{MW}.

Babson, Hoffman and Kahle~\cite{BHK11} 
studied the fundamental group of $\cY_2(n,p)$ and showed it is non-trivial and hyperbolic if $p=n^{-1/2+o(1)}$. In our terminology, they proved that with high probability $\cY_2(n,p)$ does not admit a (not necessarily stacked) contraction of a fixed $\{v_1,v_2,v_3\}$. 
In addition, 
the property of $\cY_2(n,p)$ having a triangulation with distinct labels of every $\{v_1,v_2,v_3\}$
was shown in~\cite{LuriaP18} to have the asymptotic threshold $p=(c_0 n)^{-1/2}$ for an explicit constant $c_0$.
However, it remains a challenging open problem to exclude the presence of a contraction with repeated labels of $\{v_1,v_2,v_3\}$ in $\cY_2(n,p)$ where $p=\Theta(n^{-1/2})$, thereby improving the lower bound for simple-connectivity of $\cY_2(n,p)$. In our analysis of stacked triangulation, the main challenge is again controlling the complicated effect of repeated labels.

It is worthwhile mentioning that in higher dimensions $d\ge 3$, determining whether all of the $d$-simplices admit a (not necessarily stacked) contraction in $Y\sim \cY_d(n,p)$ is purely algebraic: the homotopy groups $\pi_i(Y)$ are trivial for all $0\le i <d-1$, and therefore, by the Hurewicz Theorem, $\pi_{d-1}(Y)$ coincides with the well-studied homology group of $Y$ ~\cite{NP18}. In particular, such contractions do exist with high probability if $p\gg d\frac{\log n}n$, starkly below the critical threshold found in Theorem \ref{thm:1}.

One of the most useful techniques in studying topological invariants of $\cY_d(n,p)$ is to simplify it via elementary collapses. The threshold probability for $d$-collapsibility  of $\cY_d(n,p)$ --- the property that all $d$-faces of the complex can be removed by a sequence of elementary collapses -- was studied in~\cite{ALLM13}. Subsequently, it was discovered in~\cite{AL15,LP16} that, throughout the entire regime $p=\Theta(1/n)$, $d$-collapses play a decisive role in computing the Betti numbers of $\cY_d(n,p)$ (see the survey \cite{LP_around}). Elementary collapses were also studied on other models of random complexes such as the multi-parameter model~\cite{CF15} and the random clique complex~\cite{MA19}. The study of elementary expansions in random complexes is not as developed, but Kahle, Paquette and Rold\'an ~\cite{KPR21} used a cubical analog of elementary expansions to determine the sharp threshold probability for homological and simple connectivity of $\mathcal Q_2(n,p)$ --- a random $2$-dimensional cubical subcomplex of the $n$-dimensional cube.
\label{subsection:complexes}

\subsubsection{Hypergraph bootstrap percolation}
\label{subsection:bootstrap}
Bollob\'as introduced hypergraph bootstrap percolation in~\cite{Bol68} under the notion of weakly saturated graphs; for some of the many related notable results in extremal combinatorics, see, e.g.,~\cite{Alon85,Frankl82,Kalai84weakly}. The name $H$-bootstrap percolation was given in~\cite{BBMR12}, taking after a related model known as $r$-neighbor bootstrap percolation, introduced by the physicists Chapula, Leath and Reich~\cite{CLR79}. The latter, in which a vertex of a graph becomes infected if it has $r$ or more infected neighbors, was studied extensively (see~\cite{AL88,BBDM12} for example). 
Balogh, Bollob\'as and Morris~\cite{BBM} were the first to study $H$-bootstrap percolation started at the Erd\H{o}s--R\'enyi random graph $\cG(n,p)$; they focused on $H=K_r$, a clique on $r$ vertices, for $r\ge 4$, and found the corresponding critical probability $p_c(K_r)$  up to a poly-logarithmic factor. Thereafter, Bartha and Kolesnik~\cite{BK20} determined $p_c(K_r)$ up to a constant factor for every $r\ge 5$. Angel and Kolesnik~\cite{AK18,Kol17} investigated the special case $r=4$, and obtained the sharp asymptotics $p_c(K_4)\sim(3n\log n)^{-1/2}$. Other questions regarding graph bootstrap percolation and weak saturation were studied, e.g., in~\cite{GKP17,KS17,BC19}. A random $H$-bootstrap percolation process in which the infection of a hyperedge missing from a copy of $H$ occurs randomly, rather than deterministically, was studied in~\cite{MN21}. From this result, one can deduce an upper bound of the form $p_c(K_{d+2}^{d+1}) < C_d n^{-1/d}$, off by a constant factor; no lower bound of the correct order was previously known.

\subsection{Open problems}
Theorems~\ref{thm:1} and~\ref{thm:density} establish the phase transition for the event that $Y\sim\cY_d(n,p)$ admits a stacked contraction of $\bo=\{1,...,d+1\}$---and more generally for the behavior of the density $\bar X$ of simplices having a stacked contraction---identifying its asymptotic location to be $p\sim (\alpha_dn)^{-1/d}$. Here are some concrete problems aiming to investigate this phase transition more closely.

\begin{question}
Find $p(n)$ such that, with high probability, $Y\sim\cY_d(n,p)$ admits a stacked contraction but not a stacked triangulation (in which the labels are distinct) of $\bo$. Numerical experiments (cf.\ Figure~\ref{fig:tri-n200}) suggest such a probability $p$ exists.
\end{question}

\begin{question}
Find $w(n)$ such that at $p = (\alpha_d n)^{-1/d}+ w(n)$ the probability 
that $Y\sim\cY_d(n,p)$ admits a stacked contraction of $\bo$ converges to $1/2$. 
If the stacked contraction does exist at this $p$, how many labels does it use?
\end{question}
\begin{question}
Let $\pi$ be a uniform random ordering of the elements of $\binom{[n]}{d+1}$. Denote by $\hat X(t)$ the fraction of simplices in $\binom{[n]}{d+1}$ that have a stacked contraction whose faces are among the first $t$ elements of $\pi$. Is there, with high probability, a critical step $\tau\approx (\alpha_d n)^{-1/d}\binom{n}{d+1}$ such that $\hat X(\tau)=O(n^{-1/d})$ and $\hat X(\tau +1)=1$~?
\end{question}

% In every $H$-bootstrap percolation process of $(d+1)$-uniform hypergraphs, one may consider the event $\cA$ that every hyperedge $e$ is infected by a ``tree-like" process. That is, we can start with every hyperedge $e$ and expose, top-down, the copies of $H$ that are used to infect it such that $|V(H)|-(d+1)$ new vertices are introduced with every copy of $H$ we expose.
% \begin{question}
% Theorem \ref{thm:1} shows that in $K_{d+2}^{d+1}$-bootstrap percolation, the sharp probability thresholds for percolation and for the event $\cA$ coincide. For which other hypergraphs $H$ is this the case?
% \end{question}

The rest of the paper is organized as follows. Section~\ref{sec:prel} gives some necessary background material about simplicial complexes and algebraic shifting. In Sections~\ref{sec:subcritical} and~\ref{sec:supercritical} we prove the subcritical and supercritical items of Theorem~\ref{thm:1} respectively. Finally, in Section~\ref{sec:subcritical_zoom} we reexamine the subcritical regime and prove Theorem~\ref{thm:density}.

\section{Preliminaries} \label{sec:prel}

\subsection{Simplicial complexes}
A simplicial complex $Y$ is a family of sets that is closed under taking subsets. A set $f\in Y$ is called a face, and its dimension is $|f|-1$. The dimension of $Y$ is defined as the largest dimension of its faces. The set of $d$-dimensional faces in a simplicial complex $Y$ is denoted $Y_d$ and the $\ff$-vector of $Y$ is defined by $\ff_d(Y)=|Y_d|.$ We fix an order of the $0$-dimensional faces, aka vertices, of $Y$, and define the $d$-dimensional boundary operator $\partial_d=\partial_d(Y):\R^{Y_d}\to \R^{Y_{d-1}}$ by mapping $f=\{v_0<\cdots<v_d\}$ to
\[
 \partial_d(f) = \sum_{j=0}^{d}(-1)^j \left(f\setminus\{v_j\}\right)\,.
\]
Here faces are identified with their characteristic vectors. A basic fact in algebraic topology is that $\partial_{d}\partial_{d+1}=0$, and the $d$-th real Betti number of $Y$ is defined by 
\[ \beta_d=\beta_d(Y)=\dim{(\ker(\partial_d) / \mathrm{im}(\partial_{d+1}))}\,.\] 
In particular, if $Y$ is a $d$-dimensional complex then $\beta_d=\dim\ker\partial_d$.

\subsection{Stacked triangulations and the Fuss--Catalan numbers}
\label{subsection:stacked}
Fix $d\ge 2$ and $f=\{1,\ldots,d+1\}$. Let $\sC_d(s)$ be the number of combinatorially distinct stacked triangulations of the $d$-simplex $f$ with $s$ {\em unlabeled} internal vertices. We have that
\[
\sC_d(0)=1,\qquad \sC_d(s)~=\sum_{s_1+\cdots+s_{d+1}=s-1}\prod_{j=1}^{d+1}\sC_d(s_j),~s\ge 1\,.
\]
This recursive relation, that also describes the number of $(d+1)$-ary trees with $s$ internal vertices and $ds+1$ leaves, defines the Fuss--Catalan numbers:
\[ \sC_d(s) = \frac1{d s+1}\binom{(d+1)s}{s}
= \Big(\frac{d+1}{2\pi d^3}+o_s(1)\Big)^{1/2} s^{-3/2} (\alpha_d)^s
\]
for $\alpha_d = (d+1)^{d+1}/d^d$ as defined in~\eqref{eq:alpha-d}.

\begin{figure}
\begin{tikzpicture}[font=\small]
\node (fig1) at (0,0) {
\includegraphics[width=0.8\textwidth]{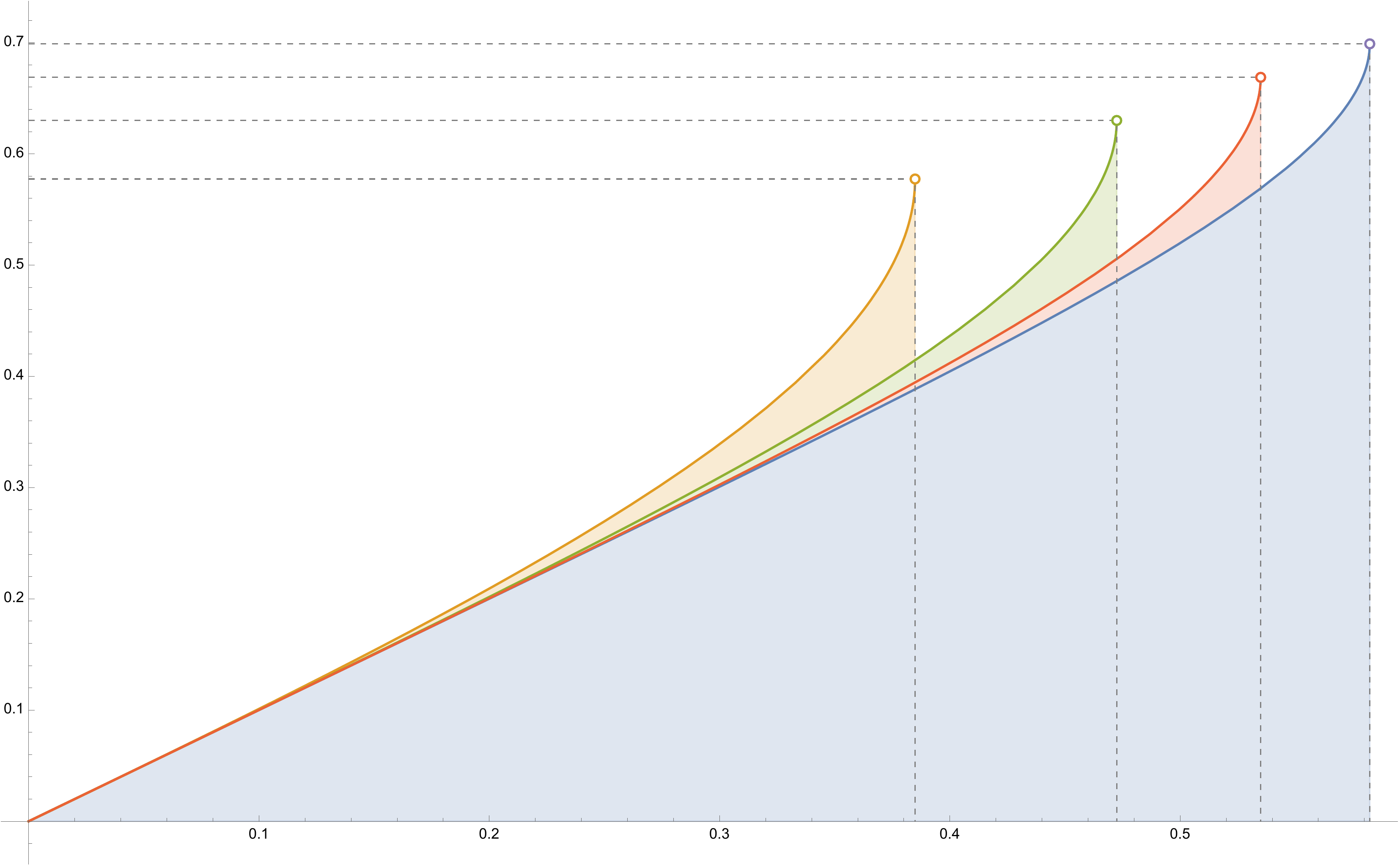}};
\node at (5.2,-2.85) {$\gamma$};
\node at (-4.9,3.4) {$f(\gamma)$};
\node[font=\tiny,color=gray] at (-5.5,1.85) {$3^{-1/2}$};
\node[font=\tiny,color=gray] at (-5.5,2.3) {$4^{-1/3}$};
\node[font=\tiny,color=gray] at (-5.5,2.6) {$5^{-1/4}$};
\node[font=\tiny,color=gray] at (-5.5,2.9) {$6^{-1/5}$};
\end{tikzpicture}
\vspace{-0.15in}
\caption{The generating function $f(\gamma)$ of the Fuss--Catalan numbers of order $d=2,3,4,5$; marked points are at  $(\alpha_d^{-1/d},(d+1)^{-1/d})$.}
\label{fig:gamma-hat}
\vspace{-0.15in}
\end{figure}
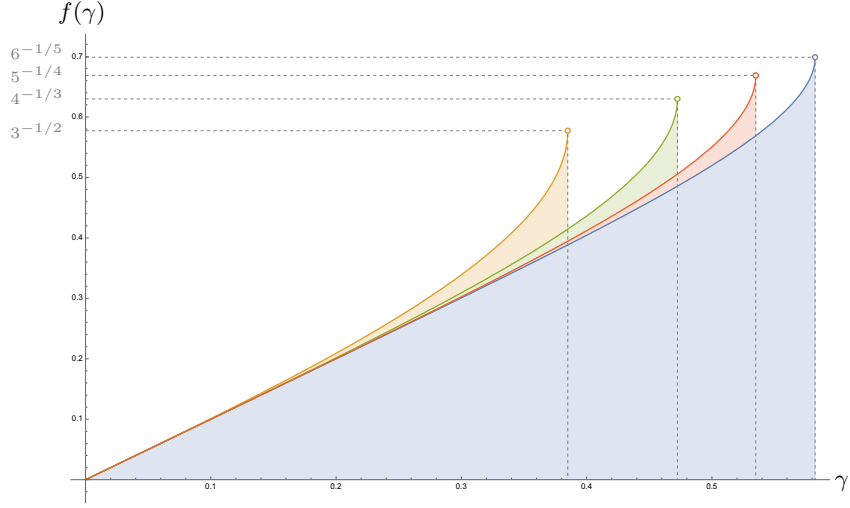

Using the recursive definition of the Fuss--Catalan number $\sC_d(s)$, it can be shown that the generating function $f(\gamma):=\sum_{s=0}^{\infty}\sC_d(s)\gamma^{ds+1}$ satisfies
\[
Q_\gamma(f(\gamma))=f(\gamma)^{d+1} - f(\gamma) + \gamma = 0,\] for every $0\le \gamma < \alpha_d^{-1/d}$. In fact, since $f(\gamma)$ is non-negative, continuous and monotone increasing in $\gamma$, it is the smaller among the two positive roots of this equation. See Figure~\ref{fig:gamma-hat} for an illustration of $f(\gamma)$ for $d=2,3,4,5$.

\subsection{Algebraic shifting}
Algebraic shifting, that  was introduced by Kalai~\cite{kalai1984,kalai1986}, is a procedure that, given a simplicial complex $Y$, produces another simplicial complex $\Delta(Y)$ with special and useful properties. There are two main variants of algebraic shifting: exterior and symmetric. Both variants can suit our purposes in this paper, and we choose to work with exterior shifting. For an extensive overview, we refer the reader to the interesting exposition in~\cite{kalai2002}. 

Let $K\subseteq\binom{[n]}{d+1}$ and $X$ be an $n\times n$ real matrix with generic entries.  Let $M$ be a $|K| \times \binom n{d+1}$ matrix, whose rows are indexed by $K$ and columns are indexed by $\binom{[n]}{d+1}$ ordered lexicographically. Each entry of $M$ is equal to the corresponding $(d+1) \times (d+1)$ minor of $X$. To construct $\Delta(K)$ we greedily choose a basis for the column space of $M$ by taking the columns that are not spanned by their predecessors in the lexicographic order, and add the indices of these columns to $\Delta(K)$. It is clear that $|\Delta(K)|=|K|$. 
In addition, $\Delta(K)$ is shifted, i.e., if $f\in\Delta(K)$, $u\notin f$, $v\in f$ and $u<v$ then $f\setminus\{v\}\cup\{u\}\in\Delta(K)$. We write $f'\leqP f$ if $f'$ can be obtained from $f$ by a sequence of shifts --- in which an element is replaced by a smaller one. In other words, $\{u_0<\cdots<u_{d}\} \leqP \{v_0<\cdots<v_{d}\}$ if and only if $u_i\le v_i~\forall i$. Note that the lexicographic order  is a linear extension of $\leqP$.

To compute the algebraic shifting of a simplicial complex $Y$ the above procedure is performed independently in each and every dimension, i.e., $\Delta(Y)=\bigcup_d \Delta(Y_d)$. Bj{\"o}rner and Kalai~\cite{BK88} showed that $\Delta(Y)$ is a simplicial complex with the same $\ff$-vector and the same Betti numbers as $Y$. Moreover, the Betti numbers of $\Delta(Y)$ are given combinatorially by
\[
\beta_d(Y)=\beta_d(\Delta(Y)) = b_d(\Delta(Y)):=|\{f\in\Delta(Y)_d~:~f\cup\{1\}\notin  \Delta(Y)\}|\,.
\]
Note that if $Y$ is a $d$-dimensional complex then so is $\Delta(Y)$ and
\[
b_d(\Delta(Y))=|\{f\in\Delta(Y)_d~:~1\notin f\}|\,.
\]

\section{The subcritical regime}
\label{sec:subcritical}
Our goal in this section is to establish the following result on the stacked contractions of $\{1,2,\ldots,d+1\}$ in $\cY_d(n,p)$, including those that allow repeated labels.

\begin{theorem}
\label{thm:subcrit-main}
Fix $d\geq 2$ and let $p = (1-\delta)(\alpha_d n)^{-1/d}$ for some $\delta=\delta(n) \in [0,\frac12]$. 
Let $\cE_s^\bo$ be the event that $\cY_d(n,p)$ admits a stacked contraction of $\bo=\{1,2,\ldots,d+1\}$ that features $s+d+1$ distinct labels. Then for every $0\leq s \leq (\log n)^{5/2}$,
\[ \P(\cE_s^\bo) \leq e^{-\delta d s} \big(p + n^{-2/d+o(1)}\big)\,.\]
\end{theorem}
The lower bound in our main theorem will follow immediately from the above theorem in light of the following simple fact.

\begin{observation}\label{obs:reduce-labels}
Suppose $K\subseteq \binom{[n]}{d+1}$ admits a stacked contraction of a face $f$ using $s+d+1$ distinct labels. Then for every $0 \leq k < s$, there is some $f'$ which has a stacked contraction in $K$ using $s'+d+1$ distinct labels, where $k/(d+1) \leq s' \leq k$. 
\end{observation}

\begin{proof}
Take $s\geq 1$ (or there is nothing to show), and consider a stacked contraction for $f$, where $z\in[n]\setminus f$ subdivides $f$ into $f^{(1)},\ldots,f^{(d+1)}$, and each $f^{(i)}$ has a stacked contraction in $K$ via $s^{(i)}+d+1$ distinct labels. Then $s \leq \sum s^{(i)} + 1$ (the last term accounting for the label $z$), hence there exists an $f^{(i)}$ with $s^{(i)} \geq (s-1)/(d+1) $ and at least one fewer internal vertex in its underlying triangulation than $f$. Set $f_0=f$ and $s_0=s$, and let $f_i$ be the face obtained in this manner from $f_{i-1}$; namely, it has
$m_i$ internal vertices and $s_i+d+1$ distinct labels in its stacked triangulation in~$K$, and
\[ \frac{s_{i-1}-1}{d+1}\leq s_i \leq s_{i-1}\qquad,\qquad m_i < m_{i-1}\,.\]
Let $t\geq 1$ be the minimal index such that $s_t \leq k$ (since the $m_i$'s are strictly decreasing, and $m_i=0$ would imply $s_i=0$, we have $t \leq m_0$).
Then we also have that $s_t \geq (s_{t-1}-1)/(d+1) \geq k/(d+1)$ by definition, thus $f'=f_t$ is as required.
\end{proof}

\begin{corollary}\label{cor:subcrit}
Fix $d\geq 2$ and let $p = (1-\frac1{\log n})(\alpha_d n)^{-1/d}$. Then the probability that $\cY_d(n,p)$ admits a stacked contraction of $\{1,2,\ldots,d+1\}$ is at most $n^{-1/d+o(1)}$.
\end{corollary}
\begin{proof}
A union bound on the estimate from Theorem~\ref{thm:subcrit-main} shows that 
\[ \P\left( \bigcup\{\cE_s^\bo \,:\; s\leq (\log n)^{5/2}\}\right ) \leq n^{-1/d+o(1)}\,.\]
Furthermore, if $\cE_s^f$ is the analog of $\cE_s^\bo$ corresponding to a stacked contraction of a face $f=\{x_1,\ldots,x_{d+1}\}$ (rather than $\bo=\{1,\ldots,d+1\}$) then 
\begin{equation}
\label{eq:cE_s^f}
 \P\left( \bigcup\{\cE_s^f \,:\; f \in \tbinom{[n]}{d+1}\,,\, \tfrac1{d+1} (\log n)^{5/2} \leq s\leq (\log n)^{5/2}\}\right ) \leq n^{-(\frac d{d+1}-o(1)) \sqrt{\log n}}\,,
\end{equation}
as the probability estimate from Theorem~\ref{thm:subcrit-main} is  at most $n^{-\frac d{d+1}\sqrt{\log n}}$ for any $f,s$ as above. 
Combining these estimates with Observation~\ref{obs:reduce-labels} concludes the proof.
\end{proof}

Theorem~\ref{thm:subcrit-main} will be derived by a first moment argument on the number of stacked contractions of $\{1,2,\ldots,d+1\}$. It is natural to view a stacked contraction of a $d$-simplex $f$ whose vertices have distinct labels as a $(d+1)$-ary tree of $d$-simplices:  the root of the tree is $f$, and the $(d+1)$ children of every node in the tree are the  $d$-simplices that subdivided it during the recursive construction of the stacked triangulation. In consequence, the leaves of this tree are the $d$-faces of the stacked triangulation. More generally, an arbitrary stacked contraction can be viewed as a directed acyclic graph (DAGs) whose vertices are a subset of $\binom{[n]}{d+1}$, as follows.

\begin{definition}[an $(m,l,s)$-pedigree] Let $G$ be a DAG on a vertex set $V\subseteq \binom{[n]}{d+1}$. We say that $G$ is an $(m,l,s)$-pedigree for a vertex $\bo\in V$ if the following hold:
\begin{enumerate}[(1)]
    \item{}[root] $\bo$ is the unique vertex in $G$ with in-degree $0$.
    \item{}[out-degrees] $G$ has $m+l$ vertices, of which $m$ have out-degree $d+1$ (internal) and the remaining $l$ have out-degree $0$ (leaves). 
    \item{} [labels] the cardinality of $\bigcup_{\bu\in V} \bu$ (all labels from $[n]$ appearing in $G$) is $s+d+1$.
    \item{}[stacking] If a vertex $\mathbf u =  \{x_1<x_2<\ldots<x_{d+1}\}$ has outgoing edges to $\mathbf v_1,\ldots,\mathbf v_{d+1}$, then there is some $z \notin \bu$ so that $\mathbf v_i = \bu \setminus \{x_i\}\cup \{z\}$ for all $i$. The label $z$ is called is the outgoing label of $\bu$. 
\end{enumerate}
\label{def:mls}
\end{definition}

One can easily verify (by the recursive definition of a stacked triangulation) that if a simpicial complex $Y$ admits a stacked contraction of a $d$-simplex $f$, then there exists a DAG as above whose leaves are all members of $Y$, and whose root $\bo$ is equal to $f$.

A crucial ingredient in the proof of Theorem~\ref{thm:subcrit-main} will be the following result, which bounds the tree excess of an $(m,l,s)$-pedigree in terms of the surplus in $s$, its total number of labels.
\begin{theorem}
\label{thm:a-b}
Let $d\geq 2$. For every $(m,l,s)$-pedigree $G$, if $\aG$ is given by
\begin{equation}\label{eq:a-def}
\aG := l - (d s + 1)
\end{equation}
(the excess in leaves) and $\bG$ is the tree excess
\begin{equation}\label{eq:b-def}
\bG := |E(G)| - (|V(G)|-1)  = d m - (l - 1)\,,
\end{equation}
then $\aG\geq 0$ and
\begin{equation}\label{eq:a-b-bound} \bG \leq d \aG^{(d+1)/d} + (d^2 - 1) \aG\,.\end{equation}
\end{theorem}

\begin{remark}\label{fig:the-frustrating-construction}
The bound $\bG \leq c_d \aG^{(d+1)/d}$ in~\eqref{eq:a-b-bound} is tight up to the value of the constant prefactor $c_d$, as one may consider, for any fixed $d\geq 2$ and any $k> d+1$, the $(m,l,s)$-pedigree $G_k$ where the $j$-th level of the breadth-first-search tree from $\bo$ consists of $\{j\} \cup f$ for all $f\in\binom{[j-1]}{d}$. Here $m = \sum_{j\leq k-2} \binom{j}{d}$, $l = \binom{k-1}{d+1}$, $s=k - (d+1)$, and $\bG$ is of order $\aG^{(d+1)/d}$.
(Figure~\ref{fig:b-a-3/2} illustrates this for $d=2$.) 
\end{remark}

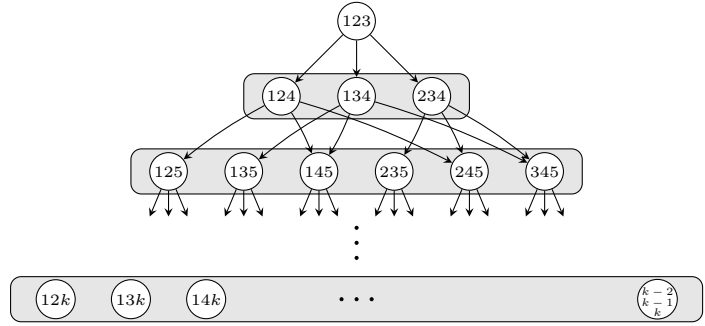
\begin{figure}
    \centering
    \vspace{-0.12in}
    \begin{tikzpicture}[every node/.style={draw, circle, font=\tiny, inner sep=1pt, fill=white}, edge from parent/.style={draw, -stealth}, level distance = 30pt]
    \node (v123) {123};
    \draw[rounded corners, fill=gray!20] (-1.5,-0.7) rectangle (1.5,-1.3);
    \draw[rounded corners, fill=gray!20] (-3,-1.7) rectangle (3,-2.3);
    \draw[rounded corners, fill=gray!20] (-4.6,-3.4) rectangle (4.6,-4);
    \foreach \i\v in {-1/124,0/134,1/234}
    { \node (v\v) at (\i,-1){\v};  }
    \foreach \i\v in {-2.5/125, -1.5/135, -0.5/145, 0.5/235, 1.5/245, 2.5/345}
    {   \node (v\v) at (\i,-2){\v};
        \foreach \w in {-0.25, 0, 0.25}    
        { \draw[-stealth] (v\v) -- +(\w,-0.6); }
    }
    \foreach \i in {-2.75, -2.95, -3.15}
    { \node[scale=0.4,fill=black] at (0,\i) {}; }
    \node at (-4,-3.7) {$12k$};
    \node at (-3,-3.7) {$13k$};
    \node at (-2,-3.7) {$14k$};
    \node[scale=5.3] at (4,-3.7) {};
    \node[draw=none,fill=none,scale=0.7] at (4,-3.6) {$k-2$};
    \node[draw=none,fill=none,scale=0.7] at (4,-3.75) {$k-1$};
    \node[draw=none,fill=none,scale=0.7] at (4,-3.87) {$k$};
    \foreach \i in {-0.2,0,0.2}
    { \node[scale=0.4,fill=black] at (\i,-3.7) {}; }
    \foreach \i in {124, 134, 234}
    { \path[-stealth] (v123) edge (v\i); }
    \path[-stealth] (v124) edge [bend right=5] (v125);
    \path[-stealth] (v124) edge [bend left=5] (v145);
    \path[-stealth] (v124) edge [bend left=5] (v245);
    \path[-stealth] (v134) edge [bend right=5] (v135);
    \path[-stealth] (v134) edge [bend left=5] (v145);
    \path[-stealth] (v134) edge [bend left=5] (v345);
    \path[-stealth] (v234) edge [bend left=5] (v235);
    \path[-stealth] (v234) edge [bend left=5] (v245);
    \path[-stealth] (v234) edge [bend left=5] (v345);
    \end{tikzpicture}
    \vspace{-0.07in}
    \caption{A family of $(m,l,s)$-pedigrees $\{G_k\}_{k\geq 3}$ for $d=2$  where $\bG$ is of order $\aG^{3/2}$ (namely, 
    $\bG\sim m=\sum_{j=2}^{k-2} \binom{j}2$, $\aG \sim l=\binom{k-1}2$ and $s=k-3$).}
    \vspace{-0.07in}
    \label{fig:b-a-3/2}
\end{figure}

\subsection{Proof of Theorem~\ref{thm:a-b} via algebaric shifting}
Whereas the fact that $\bG\geq 0$, that is, $dm \geq l-1$, is trivial from the representation of $\bG$ as the tree excess of (the undirected underlying graph of) $G$, the fact that $\aG\geq 0$ is less obvious. 
We note that it can be obtained by a direct adaptation of the methods in~\cite{BBM}. An additional proof of this fact appears in Proposition~\ref{prop:a-geq-0}, which has the benefit of also expressing $\aG$ in terms of the rank of an explicit matrix. Unfortunately, both these arguments do not provide any quantitative information. However, the matrix in Proposition~\ref{prop:a-geq-0} is a rigidity matrix, which suggests the relevance of algebraic shifting to this problem. Indeed, we will eventually obtain the fact that $\aG\ge0$ as a byproduct of the argument establishing~\eqref{eq:a-b-bound} via algebraic shifting.  

\begin{claim}\label{clm:Delta(K)}
Let $K \subseteq \binom{[s+d+1]}{d+1}$ be the 
vertex set of an $(m,l,s)$-pedigree with $s\geq 1$, and let $\bar K$ denote the $d$-dimensional complex consisting of $K$ and all of its subsets. Then the shifted complex $\Delta(\bar K)$ contains $\bv_0 := [d+1]\cup\{s+d+1\}\setminus\{2\}$.
\end{claim}
\begin{proof}
We will establish the claim for every $(m,l,s)$-pedigree $G$ with $s\geq 1$ by induction on the number of vertices $m+l$, with the base case of $m=1$ (the root) and $l=d+1$ (the case where $s=1$) containing $\bv_0$ as one of the leaves. Since $\bar K$ is in this case the boundary of a $(d+1)$-simplex, we have $\Delta(\bar K)=\bar K = \binom{[d+2]}{d+1}\ni\bv_0$.

For a general $(m,l,s)$-pedigree $G$ with $m+l>d+2$ (and $s\geq 2$), consider some vertex $\bu\in V(G)$ such that all of its $d+1$ children in $G$ are leaves (such a vertex exists in every finite DAG). Let $z$ be the outgoing label of $\bu$, let $G'$ be the DAG obtained by deleting $\bu$ and its $d+1$ children from $G$, and---denoting its vertex set by $K' \subseteq \binom{[s+d+1]}{d+1}$---let $\bar K'$ be the complex consisting of $K'$ and all of its subsets.

If the label $z$ still appears in $G'$ (i.e., $G'$ is an $(m-1,l-(d+1),s)$-pedigree), then $\bv_0\in \Delta(\bar K')$ by the induction hypothesis. In this case, the induction step is a direct consequence of the following monotonicity property of the shifted complex:
\begin{fact}\label{fact:Delta-mon}
Let $\bar K'\subseteq \bar K $ be such that $\ff_0(\bar K')=\ff_0(\bar K)$. Then $\Delta(\bar K')\subseteq \Delta(\bar K)$. 
\end{fact}
To verify this property, recall the algorithm for generating $\Delta(\bar K)$ by testing the appropriate columns of a $\ff_d(\bar K)\times \binom{\ff_0(\bar K)}{d+1}$ matrix $M$ (recall $\ff_d(\bar K) = |K|$) for linear dependence on their predecessors.
Having $\ff_0(\bar K')=\ff_0(\bar K)$ implies that the matrix $M'$ used to generate $\Delta(\bar K')$ in this manner is the submatrix of $M$ corresponding to a $|K'|$-subset of its rows; thus, if $f\notin \Delta(\bar K)$, then its corresponding column in $M$ is linearly dependent on its predecessors, so this also holds for $M'$, and $f\notin \Delta(\bar K')$.

It remains to treat the case where the label $z$ does not appear in $G'$. By permuting the labels in $G$ we may assume without loss of generality that $z= s+d+1$. Since $G'$ has $s-1\geq 1$ distinct labels, our induction then asserts that $\bv_0' \in \Delta(\bar K')$ for $\bv_0' = [d+1]\cup\{s+d\}\setminus\{2\}$.

Write $\bu=\{x_1<\ldots<x_{d+1}\}$, and define
\[ L = \{\bu\} \cup \left\{\bu\cup \{z \}\setminus\{ x_i\} \right\}\,, \]
so that $K = K' \cup L$. As usual, write $\bar L$ for the complex obtained by the closure of $L$ w.r.t.\ subsets. We will appeal to a useful characterization due to Nevo~\cite{Nevo05} for the faces of the algebraic shift of a union of two complexes:
\begin{theorem}[{\cite[Thm~4.6]{Nevo05}}]\label{thm:nevo}
Let $\Gamma_1,\Gamma_2$ be two $d$-dimensional simplicial complexes with $\ff_d(\Gamma_1\cap \Gamma_2) = 1$ (i.e., their intersection contains a single $d$-dimensional face). Then for every $f=\{y_1<\ldots<y_{d+1}\}\subseteq (\Gamma_1\cup\Gamma_2)_0$,
\[ f\in \Delta(\Gamma_1\cup \Gamma_2)\quad\Longleftrightarrow\quad y_{d+1}-y_d \leq \cD_{\Gamma_1}(f)+\cD_{\Gamma_2}(f)-D_{\Gamma_1\cap\Gamma_2}(f)\,,
\]
in which $\cD_\Gamma(f) = \#\left\{ w\,:\; \{y_1<\ldots<y_d<w\}\in \Delta(\Gamma)\right\}$.
\end{theorem}
In our application, let $\Gamma_1$ and $\Gamma_2$ correspond to $\bar K'$ and $\bar L$, respectively, noting that $K' \cap L = \{ \bu\}$ by the assumption that $z$ does not appear in $G'$. Further take $f=\bv_0$, whence $y_{d+1}-y_d = s$, and $\cD_\Gamma(\bv_0)$ counts the number of faces of $\Gamma$ of the form $[d+1]\cup\{w\}\setminus \{2\}$ for some $w>d+1$. Then
\begin{enumerate}[(i)]
    \item $\cD_{\bar K'}(\bv_0) = s-1$ (corresponding to $w=d+1+i$ for $ i \in [s-1]$, as $\bv_0'\in \Delta(\bar K')$);
    \item $\cD_{\bar L}(\bv_0) = 1$ (since $\Delta(\bar L) $ is the complete simplicial complex $\binom{[d+2]}{d+1}$);
    \item $\cD_{\bar K'\cap \bar L}(\bv_0)=0$ (since the shifting of a single $d$-face $\bu$ is $[d+1])$).
\end{enumerate}
We see that $\cD_{\bar K'}(\bv_0)+\cD_{\bar L}(\bv_0)-D_{\bar K'\cap \bar L}(\bv_0) = s = y_{d+1}-y_d$, and thus $\bv_0\in\Delta(\bar K)$ by the above theorem, as required.
\end{proof}

\begin{proof}[Proof of Theorem~\ref{thm:a-b}] Assume $s\geq 1$ (otherwise $l=1$ and $\aG=\bG=0$), and that without loss of generality the outgoing labels in $G$ are taken from $[s+d+1]$ (recalling that $\aG$ and $\bG$ are invariant under replacing every occurrence of an outgoing label $z$ by a previously unused $z'$).
Let $\bar K$ denote the complex containing the vertices of $G$ and all of their subsets, and
consider the Betti number $\beta_d:=\beta_d(\bar K) = \beta_d(\Delta(\bar K))$. 

On one hand, we claim that $\beta_d \geq m$. This follows from the basic fact that if $\bu$ is an internal vertex in $G$ with children $\bv_1,\ldots,\bv_{d+1}$, then $\partial_d \bu$ is a linear combination of $\{\partial_d \bv_i\,:\; i\in[d+1]\}$. To see this, consider an internal vertex $\bu=\{x_1<\ldots<x_{d+1}\}$, let $z$ be the outgoing label of $\bu$, and suppose that $k$ is the index such that $x_{k-1}<z<x_k$ (setting $k=0$ if $z<x_1$ and $k=d+2$ if $z>x_{d+1}$). 
 If $\bw := \{x_1<\ldots<x_{k-1}<z<x_k<\ldots<x_{d+1}\}$ then the well-known property $\partial_d\partial_{d+1}=0$ produces the sought linear combination, since 
\[ 0 = \partial_d\partial_{d+1} \bw = \sum_{i=1}^{k-1} (-1)^i \partial_d \bv_i
+ (-1)^k \partial_d \bu + \sum_{i=k}^{d+1} (-1)^{i+1} \partial_d \bv_i\,.\]
With this fact in hand, one can construct the matrix corresponding to $\partial_d$ as follows: begin with the columns $\partial_d \bv_1,\ldots,\partial_d \bv_l$ where $\bv_1,\ldots,\bv_l$ are the leaves of $G$; then repeatedly add the column $\partial_d \bu$ for an internal vertex $\bu$ all of whose children were all added. As each such iteration adds a column which is a linear combination of its predecessors, we deduce that
\begin{equation}\label{eq:beta-2-lower}
\beta_d = \dim \ker(\partial_d) \geq m\,.\end{equation}
(Note that this inequality may be strict: see Figure~\ref{fig:beta_d-m}.)

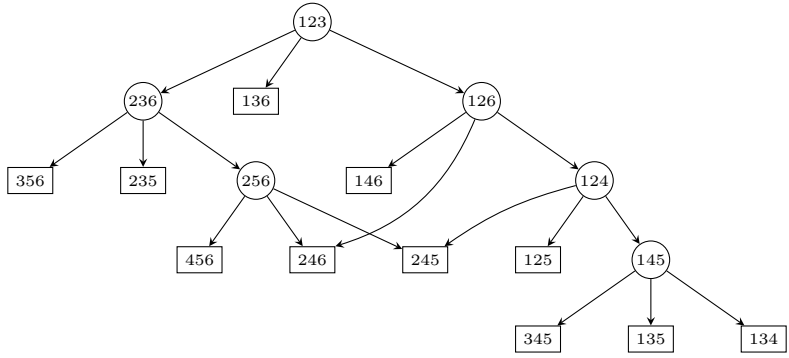
\begin{figure}
    \centering
    \begin{tikzpicture}[every node/.style={draw, circle, font=\tiny, inner sep=1pt}, 
    edge from parent/.style={draw, -stealth}, 
    level distance = 30pt]
    \node{123}
        child {node {236} 
            child {node[rectangle, inner sep=3pt] {356}}
            child {node[rectangle, inner sep=3pt] {235}}
            child {node {256}
                child[missing]
                child {node[rectangle, inner sep=3pt] {456}}
                child {node[rectangle, inner sep=3pt] (v246) {246}}
                child {node[rectangle, inner sep=3pt] (v245) {245}}
            }
        }
        child {node[rectangle, inner sep=3pt] {136}}
        child [missing]
        child {node (v126) {126} 
            child {node[rectangle, inner sep=3pt] {146}}
            child[missing]
            child {node (v124) {124}
                child {node[rectangle, inner sep=3pt] {125}}
                child {node {145}
                    child {node[rectangle, inner sep=3pt] {345}}
                    child {node[rectangle, inner sep=3pt] {135}}
                    child {node[rectangle, inner sep=3pt] {134}}
                }
            }
        }
        ;
    \path[-stealth] (v124) edge [bend right=10] (v245);
    \path[-stealth] (v126) edge [bend left=25] (v246.north east);
    \end{tikzpicture}
    \caption{A $(6,11,3)$-pedigree $G$ for $d=2$ whose simplicial complex has $\beta_2 > m$: as the dimension of the leaves is at most $\binom{s+d}2$, in this case it is at most $10 < l$, not linearly independent (here $\beta_2=7$ vs. $m=6$).}
    \label{fig:beta_d-m}
\end{figure}
On the other hand, an upper bound on $\beta_d$ can be derived from its characterization 
\[ \beta_d = b_d(\Delta(\bar K)) = |L|\quad\mbox{where}\quad L := \{f\in\Delta(\bar K) \,:\; |f|=d+1\,,\,1\notin f\}\,.\] 
Recalling that  $\bv_0\in\Delta(\bar K)$ by Claim~\ref{clm:Delta(K)},
define
\begin{align*}
F &:= \big\{ f\in\tbinom{[s+d+1]}{d+1}\,:\; f \leqP \bv_0\,,\, 1\in f \big\} \subseteq \Delta(\bar K)\qquad\mbox{and}\\ 
\widehat F  &:= \big\{ f\in \Delta(\bar K)\,:\; |f|=d+1\,,\,1\in f \big\}\setminus F\,,\end{align*}
where $F\subseteq \Delta(\bar K)$ as per the definition of a shifted complex containing~$\bv_0$, and set
\[ \aGhat := |\widehat F| = \ff_d(\Delta(\bar K))-\beta_d - |F|\,.\]
We readily see that 
\[ F =  \left\{ [d+1]\right\} \cup \left\{ [d+1]\cup\{x+d+1\}\setminus\{y\} \,: x\in[s]\,,\,y=2,\ldots,d+1\right\}\,,\]
so that \[|F|=ds+1\]
and substituting this (as well as $|\Delta(\bar K)|=m+l$) in the expression for $\aGhat$ we get
\begin{equation}\label{eq:a-a'}
\aGhat = l - (ds+1) - (\beta_d-m) = \aG - (\beta_d-m) \leq \aG,,
\end{equation}
with the last inequality using~\eqref{eq:beta-2-lower}. In particular, $\aG \geq |\widehat F|\geq 0$.

We next aim to bound $|L|$ in terms of $\aGhat$. Write $L$ as the disjoint union of
\begin{align*}
    L_1 &= \left\{ \{x_1<\ldots<x_{d+1}\}\in L \,:\; x_d\leq d+1\right\}\,,\\
    L_2 &= \left\{ \{x_1<\ldots<x_{d+1}\}\in L \,:\; x_d> d+1\,, x_1 \leq d+1\right\}\,,\\
    L_3 &= \left\{ \{x_1<\ldots<x_{d+1}\}\in L \,:\; x_1\geq d+2\right\}\,.
\end{align*}
Recalling $1\notin f$ for $f\in L_1$, clearly $L_1$ is $\{ (2,\ldots,d+1,x+d+1) \,:\; x\in[s]\}$, and so
\[ |L_1| = s\,.\]
To treat $L_2$, consider the map $T$ which replaces the smallest element of $f\in L_2$ by $1$. Since every $f\in L_2$ has at least $2$ elements that exceed $d+1$, we see that~$T f \nleqP \bv_0$, whence $T L_2 \subseteq \widehat F$. Moreover, every element of $\widehat F$ has at most $d$ preimages of $T$ in $L_2$ thanks the requirement $2 \leq x_1\leq d+1$, hence
\[ |L_2| \leq d \,\aGhat\,.\]
Lastly, to handle $L_3$, set 
\[ \widehat F' = \{ f \setminus \{1\} \,:\; f \in \widehat F\}\,.\] 
For every $f=\{x_1,\ldots,x_{d+1}\}\in L_3$ and every $i$, we have that 
\[ f\setminus \{x_i\} \in \widehat F'\,,\] 
since it has $d$ elements (and in particular, at least $2$ elements) that exceed $d+1$ by the definition of $L_3$, whence $f\cup\{1\}\setminus \{x_i\}\nleqP \bv_0$.
Applying the Kruskal--Katona theorem in its simpler form due to Lov\'asz (see~\cite[p.~95, Ex.~13.31(b)]{Lovasz07}) we find that, if $|L_3| = \binom{\theta}{d+1}$ for $\theta\in\R_+$ then $\binom{\theta}{d} \leq |\widehat F'| \leq \aGhat$ (and $|L_3|=\binom{\theta}d \frac{\theta-d}{d+1}$), whence
\[ |L_3| \leq \aGhat \frac{(d! \aGhat)^{1/d}}{d+1} \leq  (\aGhat)^{(d+1)/d} \,.\]
Combining these estimates shows that $\beta_d=|L|$ satisfies
\begin{equation}\label{eq:beta-2-upper}
\beta_d \leq s + d \aGhat + (\aGhat)^{(d+1)/d)}\,.
\end{equation}
Using this estimate we may now conclude the proof. 
% We have seen that
%\[ \beta_d = |\Delta(K)|-(|F|+|\widehat F|) = m + l - (ds+1)-\aGhat = m+ \aG-\aGhat\,,\]
%from which we infer (with~\eqref{eq:beta-2-lower} in mind) that indeed $\aG\geq 0$, since
%\[ \aG = \aGhat + m-\beta_d \geq \aGhat \geq 0\,.\]
To bound $\bG = d m - (l-1) $ from above we can increase $m$ to $\beta_d$ (with~\eqref{eq:beta-2-lower} in mind) and get
\begin{align*}
     \bG &\leq d \beta_d - (l-1) = d (\beta_d - s) - (l - (ds+1)) \leq d \aG^{(d+1)/d} + (d^2-1) \aG\,,
\end{align*}
where the last inequality used~\eqref{eq:beta-2-upper} and the fact that $\aGhat \leq \aG$.
\end{proof}

\subsection{Proof of Theorem~\ref{thm:subcrit-main}}

Let $\cE_{m,l,s}$ be the event that $Y \sim \cY_d(n,p)$ admits a stacked contraction of $\bo=\{1,2,\ldots,d+1\}$ corresponding to some $(m,l,s)$-pedigree, so that $\cE_s^\bo = \bigcup_{m,l,s} \cE_{m,l,s}$.

Define $\aG,\bG$ as~\eqref{eq:a-def}--\eqref{eq:b-def}, and note that if $\aG=0$ then also $\bG=0$ by Theorem~\ref{thm:a-b} (while this follows from~\eqref{eq:a-b-bound}, it can be readily proved via an elementary argument examining the pedigree under consideration; we will need the delicate bound~\eqref{eq:a-b-bound} later in order to treat $\bG>0$). In particular, in this case $m=s$ and $l=ds+1$, and the $(m,l,s)$-pedigrees are precisely the set of $(d+1)$-ary trees with $s$ internal vertices. As the number of such trees is $\sC_d(s) \leq \alpha_d^s$, we thus have the upper bound 
\[ \P\left(\cE_{s,ds+1,s}\right) \leq \binom{n}{s} s! \, \sC_d(s) p^{ds+1} \leq (1-\delta)^{ds} p \leq e^{-\delta d s} p\,.
\]
Since $\aG \geq 0$ holds for any $(m,l,s)$-pedigree as per (the easy part of) Theorem~\ref{thm:a-b}, it remains to treat $\aG\geq 1$, which we henceforth assume, noting this implies $s\geq 2$.

Let us first narrow down the relevant values of the parameter $2 \leq s \leq (\log n)^{5/2}$ and $\aG = l-(ds+1)$ via a crude union bound that ignores the DAG structure (including its number of internal vertices) and merely looks at the probability that $Y$ contains $l$ faces using at most $s+d+1$ different labels:
\begin{align}\nonumber  \P\bigg(\bigcup_m\cE_{m,l,s}\bigg) &\leq \binom{n}{s} \binom{\binom{s+d+1}{d+1}}l p^l \leq 
\left( n^{\frac1d} p (s+d+1)^{d+1}\right)^{d s}\left( (s+d+1)^{d+1} p\right)^{\aG+1} \\
&\leq s^{c_d (s + \aG+1)} n^{-(\aG+1)/d} \leq s^{c_d s} n^{-(\frac1d-o(1))(\aG+1)}\,,
\label{eq:E(l,s)-crude}
\end{align}
for some $c_d>0$, where it was possible to increase $c (s+ c)^{c s}$ into $s^{c' s}$ having $s\geq 2$,
and we thereafter absorbed the term $s^{O(\aG)} = \exp[ O(\aG\log \log n)]$ into the $o(1)$-term.
If $s \leq \sqrt{\log n}$ then the right-hand of~\eqref{eq:E(l,s)-crude} is in this case at most
\[ \exp\left[ O(\sqrt{\log n} \log\log n) - (\aG+1)(\tfrac1d - o(1))\log n\right] \leq n^{-2/d+o(1)}\,,
\]
with the last step using the fact that $\aG\geq 1$. We may thus assume henceforth that
\begin{align}\label{eq:s-crude-bound} 
\sqrt{\log n} \leq s \leq (\log n)^{5/2}\,.
\end{align}
Reexamining~\eqref{eq:E(l,s)-crude} for this adjusted range of $s$, we see that if $\aG \geq s^{5/8}$ then
\[ s \log s \leq \aG s^{3/8} \log s \leq \aG(\log n)^{\frac{15}{16}+o(1)} = o(\aG\log n)\,,\]
so the right-hand of~\eqref{eq:E(l,s)-crude} is in this case at most $n^{-(1/d-o(1))\aG} \leq n^{-\sqrt{\log n}}$ for large enough $n$; thus, we can assume
\begin{equation}\label{eq:a-crude-bound} 
\aG \leq s^{5/8}\,.
\end{equation}
Recalling the definitions of $\aG,\bG$ from~\eqref{eq:a-def}--\eqref{eq:b-def}, we have $d (m - s) = \aG+\bG$, and so~\eqref{eq:a-b-bound} yields that, in the above range of $\aG$,
\[ 0 \leq m - s \leq d( \aG^{(d+1)/d} + \aG ) = O(\aG^{(d+1)/d}) = O(a^{3/2}) = O(s^{\frac{15}{16}}) = o(s)\]
and we deduce that 
\[ m = (1+o(1))s\,.\]
We now wish to bound $\P(\cE_{l,s})$ by enumerating over $(m,l,s)$-pedigrees. To this end, we argue that the number of unlabeled DAGs which can be obtained by erasing the vertex labels of an $(m,l,s)$-pedigree is at most
\[ ((d+1)m)^{2\bG}\sC_d(m)\,;\]
indeed, given a DAG that arises from an $(m,l,s)$-pedigree, one can repeatedly (for $\bG$ steps) take a tree excess edge and replace it by an out-edge to a new leaf, arriving at a $(d+1)$-ary tree $T$ with $m$ internal vertices (and $dm+1$ leaves). To recover the original DAG, one can repeatedly (for $\bG$ steps) choose a leaf (of those we had added,  identifying its parent as the source of an excess edge) and a vertex (the target of said edge), having 
$(dm+1)((d+1)m+1) \leq ((d+1)m)^2$ options per step (this inequality holds for all $m$ large enough); the above estimate now follows from the fact that the number of $(d+1)$-ary trees with $m$ internal vertices is $\sC_d(m)$.

For each unlabeled DAG as above, one can assign vertex labels from $\binom{[n]}{d+1}$ using a total of $s$ distinct outgoing labels, by selecting said $s$ labels from $[n]$, then choosing the $s$ internal vertices which will have those outgoing labels, and thereafter labeling the remaining $m-s$ vertices; carrying the labeling process in this way leverages the fact that $m-s=(\aG+\bG)/d$, which we have good control over.

Combining these two arguments, we conclude that
\begin{align*}
 \P(\cE_{m,l,s}) & \leq  \binom{n}{s} \binom{m}{s} s! s^{m-s} ((d+1)m)^{2\bG} \sC_d(m) p^l \\
    &\leq (\alpha_d n p^d)^{s} \binom{m}{m-s} ((d+1)m)^{2\bG} (\alpha_d s)^{(\aG+\bG)/d}p^{\aG+1}\,,
\end{align*}
where the last inequality used that $\sC_d(m) \leq \alpha_d^m 
$ and that $m=s+(\aG+\bG)/d$.
Substituting the value of $p$, we have $\alpha_d n p^d = (1-\delta)^d$, 
and using $\binom{m}{m-s} \leq m^{(\aG+\bG)/d}$ we further find that
\[
\P(\cE_{m,l,s}) \leq \frac{(1-\delta)^{ds}}{(\alpha_d n)^{1/d}} \Big(\frac {m s} n\Big)^{\aG/d} 
\Big(  \alpha_d \, (d+1)^{2d} m^{2d+1}s  \Big)^{\bG/d}\,.
\]
In what is our central application of Theorem~\ref{thm:a-b}, we can next bound $\bG$ in terms of $\aG$ and find that
\[
\bG \leq d \aG^{(d+1)/d} + d ^2 \aG \leq 2d^2 \aG s^{5/16}  \leq \aG s^{1/3}\,,\]
and plugging this in the last bound on $\P(\cE_{m,l,s})$, as well as $\alpha_d = \frac{(d+1)^{d+1}}{d^d} \leq e (d+1)$ for all $d\geq 2$, we get that
\begin{align*}
    \P(\cE_{m,l,s}) &\leq 
    \frac{(1-\delta)^{ds}}{(\alpha_d n)^{1/d}}
    \bigg[\frac{m s}n \Big( e (d+1)^{2d+1}  m^{2d+1}s\Big)^{ s^{1/3}}\bigg]^{\aG/d} \\
    &\leq \frac{(1-\delta)^{ds}}{(\alpha_d n)^{1/d}}
    \bigg[ n^{-1} \exp\left(\left(2(d+1) + o(1)\right)s^{1/3}\log s\right) \bigg]^{\aG/d} \\
    &\leq e^{-\delta d s} n^{(-\frac1d + o(1))(\aG+1)}\,,
\end{align*} 
with the last equality following from the fact that
\[ s^{1/3}\log s \leq (\log n)^{5/6 + o(1)} = o(\log n)\,.\]
We conclude with a union bound over $\bigcup_{m,l}\cE_{m,l,s}$ (going over $l$ via $\aG = l-(ds+1)$).
The probability of this union of events for a given $\aG$ and all $m \leq(1+o(1) s = n^{o(1)}$ is at most $e^{-\delta d s}n^{(-1/d+o(1))(\aG+1)}$; summing this over all $\aG\geq 1$ concludes the proof.
\qed

\section{The supercritical regime}\label{sec:supercritical}

We establish the supercritical regime in Theorem~\ref{thm:1} by proving that a stronger property occurs with high probability. Namely, that $\cY_d(n,p)$ admits a so-called balanced stacked triangulation for every $f\in\binom{[n]}{d+1}$. To be consistent with the previous section, we state this result in the language of pedigrees.

\begin{definition}
A {\em proper} $s$-pedigree is an $(m,l,s)$-pedigree $T$ in which all the outgoing labels of its internal vertices are distinct, and differ from those on the root of $T$. In such a case $m=s$, $T$ is a tree and we denote the set consisting of its $l=ds+1$ leaves by $\sL(T)$.

Given a vertex $\bv$ of $T$, let the subtree $T_{\bv}$ of $T$ rooted at $T$ be the subgraph induced by $\bv$ and all its descendants. If the numbers of leaves in each of the $d+1$ subtrees of $T$ rooted at the children of $T$'s root are equal, $T$ is called {\em balanced}. In such a case $s=(d+1)s'+1$ for some $s'$, and each such subtree has $ds'+1$ leaves.
\end{definition}

\begin{theorem}\label{thm:supcrit}
Fix $d\ge 2$ and let $\delta = 1/(18(d+1)^2)$, $\epsilon=n^{-\delta}$, $p=(1+\epsilon)(\alpha_d n)^{-1/d}$, $s'=n^{2\delta}$ and $s=(d+1)s'+1$. Let $\cA_s$ denote the event that there is a balanced proper $s$-pedigree for $\{1,\ldots,d+1\}$ whose leaves are all members of $Y\sim \cY_d(n,p)$. Then, for every sufficiently large $n$,
\[
\P(\cA_s) \ge 1 - \exp{\left(-n^{12\delta}\right)}\,.
\]
\end{theorem}
The upper bound in our main theorem will follow immediately from the above theorem by taking a union bound over the $\binom n{d+1}$ possible roots.
\subsection{Combinatorics of proper pedigrees} 
We begin the proof with some combinatorial properties of proper pedigrees. The following lemma shows that, for any proper pedigree of $\bo$ and any prescribed subset $P$ of its faces, there exists a proper pedigree of $\bo$ that contains this subset $P$ and no additional labels (see Figure~\ref{fig:bos}). This fact will play a key role in characterizing (and enumerating) the pairs of proper pedigrees with a given intersection, as part of the second moment analysis.

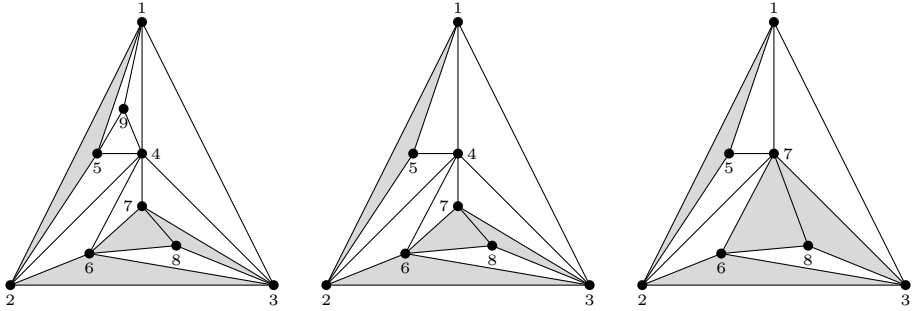
\begin{figure}
    \centering
    \begin{tikzpicture}
    \newcommand*{\tricoords}{%
    \coordinate (v1) at (0.5,1);
    \coordinate (v2) at (0,0);
    \coordinate (v3) at (1,0);
    \coordinate (v4) at (0.5,0.5);
    \coordinate (v5) at (0.33,0.5);
    \coordinate (v6) at (0.3, 0.12);
    \coordinate (v7) at (0.5, 0.3);
    \coordinate (v8) at (0.63, 0.15);
    \coordinate (v9) at (0.43,0.67);
    }
    
    \begin{scope}[scale=3.5]
    \tricoords
    % Shaded leaves
    \fill[color=gray!30] (v1)--(v5)--(v2)--cycle;
    \fill[color=gray!30] (v3)--(v2)--(v6)--cycle;
    \fill[color=gray!30] (v6)--(v7)--(v8)--cycle;
    \fill[color=gray!30] (v3)--(v7)--(v8)--cycle;
    % Edges
    \draw (v1)--(v3)--(v2)--cycle;
    % v4
    \draw (v1)--(v4)--(v3);
    \draw (v2)--(v4);
    % v5
    \draw (v1)--(v5)--(v2);
    \draw (v4)--(v5);
    % v9
    \draw (v1)--(v9)--(v4);
    \draw (v5)--(v9);
    % v6
    \draw (v4)--(v6)--(v2);
    \draw (v3)--(v6);
    % v7
    \draw (v4)--(v7)--(v6);
    \draw (v3)--(v7);
    % v8
    \draw (v6)--(v8)--(v7);
    \draw (v3)--(v8);
    % Above labels
    \foreach \i in {1} {
    \node[circle,scale=0.4,fill=black,label={[label distance=-2pt]{\tiny$\i$}}] at (v\i) {};
    }
    % Below labels
    \foreach \i in {2,3,5,6,8,9} {
    \node[circle,scale=0.4,fill=black,label={[label distance=-2pt]below:{\tiny$\i$}}] at (v\i) {};
    }
    % Right labels
    \foreach \i in {4} {
    \node[circle,scale=0.4,fill=black,label={[label distance=-2pt]right:{\tiny$\i$}}] at (v\i) {};
    }
    % Left labels
    \foreach \i in {7} {
    \node[circle,scale=0.4,fill=black,label={[label distance=-2pt]left:{\tiny$\i$}}] at (v\i) {};
    }
    \end{scope}
    
    \begin{scope}[scale=3.5,shift={(1.2,0)}]
        \tricoords
    % Shaded leaves
    \fill[color=gray!30] (v1)--(v5)--(v2)--cycle;
    \fill[color=gray!30] (v3)--(v2)--(v6)--cycle;
    \fill[color=gray!30] (v6)--(v7)--(v8)--cycle;
    \fill[color=gray!30] (v3)--(v7)--(v8)--cycle;
    % Edges
    \draw (v1)--(v3)--(v2)--cycle;
    % v4
    \draw (v1)--(v4)--(v3);
    \draw (v2)--(v4);
    % v5
    \draw (v1)--(v5)--(v2);
    \draw (v4)--(v5);
    % v6
    \draw (v4)--(v6)--(v2);
    \draw (v3)--(v6);
    % v7
    \draw (v4)--(v7)--(v6);
    \draw (v3)--(v7);
    % v8
    \draw (v6)--(v8)--(v7);
    \draw (v3)--(v8);
    % Above labels
    \foreach \i in {1} {
    \node[circle,scale=0.4,fill=black,label={[label distance=-2pt]{\tiny$\i$}}] at (v\i) {};
    }
    % Below labels
    \foreach \i in {2,3,5,6,8} {
    \node[circle,scale=0.4,fill=black,label={[label distance=-2pt]below:{\tiny$\i$}}] at (v\i) {};
    }
    % Right labels
    \foreach \i in {4} {
    \node[circle,scale=0.4,fill=black,label={[label distance=-2pt]right:{\tiny$\i$}}] at (v\i) {};
    }
    % Left labels
    \foreach \i in {7} {
    \node[circle,scale=0.4,fill=black,label={[label distance=-2pt]left:{\tiny$\i$}}] at (v\i) {};
    }
    \end{scope}
    
    \begin{scope}[scale=3.5,shift={(2.4,0)}]
        \tricoords
    % Shaded leaves
    \fill[color=gray!30] (v1)--(v5)--(v2)--cycle;
    \fill[color=gray!30] (v3)--(v2)--(v6)--cycle;
    \fill[color=gray!30] (v6)--(v4)--(v8)--cycle;
    \fill[color=gray!30] (v3)--(v4)--(v8)--cycle;
    % Edges
    \draw (v1)--(v3)--(v2)--cycle;
    % v4
    \draw (v1)--(v4)--(v3);
    \draw (v2)--(v4);
    % v5
    \draw (v1)--(v5)--(v2);
    \draw (v4)--(v5);
    % v6
    \draw (v4)--(v6)--(v2);
    \draw (v3)--(v6);
    % v8
    \draw (v6)--(v8)--(v4);
    \draw (v3)--(v8);
    % Above labels
    \foreach \i in {1} {
    \node[circle,scale=0.4,fill=black,label={[label distance=-2pt]{\tiny$\i$}}] at (v\i) {};
    }
    % Below labels
    \foreach \i in {2,3,5,6,8} {
    \node[circle,scale=0.4,fill=black,label={[label distance=-2pt]below:{\tiny$\i$}}] at (v\i) {};
    }
    % Right labels
    \foreach \i / \l in {4/7} {
    \node[circle,scale=0.4,fill=black,label={[label distance=-2pt]right:{\tiny$\l$}}] at (v\i) {};
    }
    \end{scope}
    
    \end{tikzpicture}
    \caption{An illustration of Lemma~\ref{lem:boston}: given a subset $P$ of the faces of a proper stacked triangulation of $\bo=\{1,2,3\}$ (shaded faces on left picture), there exists such a triangulation of $\bo$ containing these faces and no other labels (on right). In this example, to generate the latter,  one first deletes the vertex $9$ (middle picture), then contracts the vertices~$4,7$.}
    \label{fig:bos}
\end{figure}

\begin{lemma}\label{lem:boston}
Let $T$ be a proper $s$-pedigree for $\bo=\{x_1,\ldots,x_{d+1}\}$, $P\subseteq \sL(T)$, $R=\bigcup_{\bv\in P}\bv\setminus\bo$ and $r=|R|$. Then,
\begin{enumerate}[(1)]
    \item There exists a proper $r$-pedigree $H=H(P,T)$ for $\bo$ such that $P\subseteq \sL(H)$. Furthermore, $P=\sL(H)$ if and only if $P=\sL(T)$.
    \label{lem_item:bos_a}
    \item If $P\subsetneq \sL(T)$ then $|P|\le dr$, and equality holds if and only if $P=\sL(T)\setminus \sL(T_\bv)$ for some vertex $\bv$ of $T$.
    \label{lem_item:bos_b}
\end{enumerate}

\end{lemma}
\begin{proof}
We start with the proof of item~\eqref{lem_item:bos_a}.
Let $z$ be the outgoing label in $T$ of the root $\bo$. For every $1\le j \le d+1$, the subtree $T_j$ of $T$ rooted at $\bv_j=\bo\setminus\{x_j\}\cup\{z\}$ is a proper $s_j$-pedigree for $\bv_j$, for some $s_j < s$. 
Let $P_j:=P\cap \sL(T_j)$, $R_j=\bigcup_{\bv\in P_j}\bv\setminus\bv_j$ and $r_j=|R_j|$.
By induction on $s$, we find $H_j=H(P_j,T_j)$ --- an $|R_j|$-pedigree for $\bv_j$ such that $P_j\subseteq \sL(H_j)$, and $P_j=\sL(H_j)$ if and only if $P_j=\sL(T_j)$. The label sets $R_1,\ldots,R_{d+1}$ are disjoint and contained in $R$. If $z\notin R$ then $R=\bigcup_jR_j$ and otherwise $R=\{z\}\cup\bigcup_jR_j$. In particular, $r=\one_{\{z\in R\}}+\sum_jr_j.$

Let $H'$ be the proper $(1+\sum_jr_j)$-pedigree  for $\bo$ obtained from $T$ by replacing each $T_j$ with $H_j$. 
If $z \in R$ it is straightforward to see that setting $H$ as $H'$ concludes the proof.
Otherwise, the label $z \notin R$ needs to be removed from $H'$. We further assume that $P\neq\emptyset$ or else we set $H$ as the isolated vertex $\bo$. Without loss of generality, suppose that $P_{d+1}\ne\emptyset$.

Let $\bu=\{y_1,\ldots,y_d,z\}$ be an internal vertex in $H_{d+1}$ containing $z$ of maximal distance from $\bo$. Such an internal vertex exists because $z\notin R$ whence the root $\bv_{d+1}$ of $H_{d+1}$ is an internal vertex containing $z$. Consequently, if we denote the outgoing label of $\bu$ by $y$, then all the vertices $\bu\setminus\{y_i\}\cup\{y\},~i=1,\ldots,d$ contain $z$ and of greater distance to $\bo$ than $\bu$ and, therefore, are leaves. Let $T_{\bu'}$ be the subtree of $H_{d+1}$ rooted at $\bu'=\{y_1,\ldots,y_d,y\}$.

We construct $H=H(P,G)$ from $H'$ by the following procedure. First remove all the descendants of $\bu$ from $H'$. Afterwards, replace every vertex $\bv$ in $H'$ containing $z$ with $\bv\setminus \{z\}\cup\{y\}$. In particular, the vertex $\bu$ is replaced by $\bu'$. The final step is to append the subtree $T_{\bu'}$ to $\bu'$ (See Figure~\ref{fig:tree_procedure}).

\begin{figure}
    \centering
     \begin{tikzpicture}[every node/.style={draw, circle, font=\tiny, inner sep=1pt}, 
    edge from parent/.style={draw, -stealth}, 
    level distance = 30pt]
     \draw[rounded corners, fill=blue!1] (-4.2,-4.55) rectangle (5.7,.35);
    \node(o){123}
        child {node {124} 
            child {node[rectangle, inner sep=3pt,fill=gray!30] {125}}
            child {node[rectangle, inner sep=3pt] {145}}
            child {node[rectangle, inner sep=3pt] {245}}
        }
        child {node[rectangle, inner sep=3pt] {134}}
        child [missing]
        child {node {234} 
            child {node[rectangle, inner sep=3pt,fill=gray!30] {236}}
            child {node[rectangle, inner sep=3pt] {246}}
            child {node  {346}
                child {node[rectangle, inner sep=3pt] {347}}
                child {node {367}
                    child {node[rectangle, inner sep=3pt] {368}}
                    child {node[rectangle, inner sep=3pt,fill=gray!30] {378}}
                    child {node[rectangle, inner sep=3pt,fill=gray!30] {678}}
                }
                child {node[rectangle, inner sep=3pt] {467}}
            }
        }
        ;
        \node[draw=none,font=\small] at ($(o-4-3)+(0.38,0.05)$) {$\bu$};
        \node[draw=none,font=\small] at ($(o-4-3-2)+(0.45,0.2)$) {$\bu'$};
     \path[blue, draw, thick,dashed, rounded corners, fill=blue, fill opacity=0.05] ($(o-4)+(0,0.5)$) -- ($(o-4-1)+(-0.4,0.2)$) -- ++(0,-2.55) -- ($(o-4-3-2-3)+(0.4,-0.25)$) -- ++(0,1.55) -- cycle;
     \draw[blue, thick, dashed, rounded corners, fill=blue, fill opacity=0.05] ($(o-2)+(-0.4,-0.3)$) rectangle ($(o-2)+(0.4,0.3)$);
     \path[blue, draw, thick,dashed, rounded corners, fill=blue, fill opacity=0.05] ($(o-1)+(0,0.45)$) -- ($(o-1-1)+(-0.4,0.2)$) -- ++(0,-0.45) -- ($(o-1-3)+(0.4,-0.25)$) -- ++(0,.45) -- cycle;
    \node[draw=none,blue,font=\small] at ($(o-1)+(-0.7,0.35)$) {$H_1$};
    \node[draw=none,blue,font=\small] at ($(o-2)+(0.7,0.35)$) {$H_2$};
    \node[draw=none,blue,font=\small] at ($(o-4)+(0.7,0.35)$) {$H_3$};
    \end{tikzpicture}    
    \begin{tikzpicture}[every node/.style={draw, circle, font=\tiny, inner sep=1pt}, 
    edge from parent/.style={draw, -stealth}, 
    level distance = 30pt]
     \draw[rounded corners, fill=red!1] (-4.2,-3.5) rectangle (5.7,.35);
    \node{123}
        child {node {127} 
            child {node[rectangle, inner sep=3pt,fill=gray!30] {125}}
            child {node[rectangle, inner sep=3pt] {157}}
            child {node[rectangle, inner sep=3pt] {257}}
        }
        child {node[rectangle, inner sep=3pt] {137}}
        child [missing]
        child {node {237} 
            child {node[rectangle, inner sep=3pt,fill=gray!30] {236}}
            child {node[rectangle, inner sep=3pt] {267}}
            child {node {367}
                    child {node[rectangle, inner sep=3pt] {368}}
                    child {node[rectangle, inner sep=3pt,fill=gray!30] {378}}
                    child {node[rectangle, inner sep=3pt,fill=gray!30] {678}}
                }    }
        ;
    \end{tikzpicture}  
    \caption{An example of the construction of $H$ (bottom) from $H'$ (top) in the non-trivial case in which $P\neq\emptyset$ (shaded) and $z=4\notin R$. We use the rightmost tree to replace $z$, so, in the notations of the proof, $\bu=346$, $y=7$ and $T_{\bu'}$ is the tree descending from $\bu'=367$. To obtain $H$ we delete the descendants of $346$, replace all the appearances of $4$ by $7$, and attach $T_{\bu'}$ to the vertex $367$ that replaced $346$. Alternatively, the leftmost subtree could have been used with $\bu=124$ and $y=5$. }
    \label{fig:tree_procedure}
\end{figure}
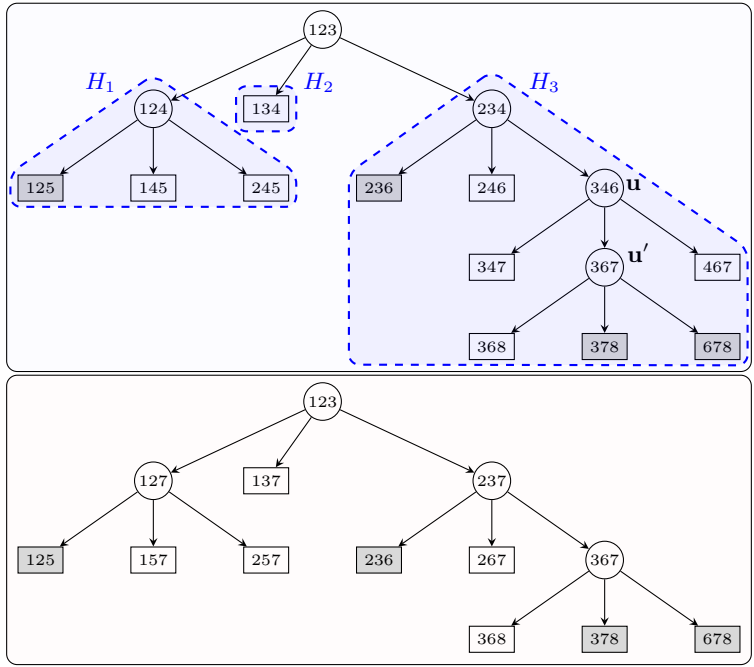

The fact that $P\subseteq \sL(H)$ follows from observing that the only leaves in $\sL(H')\setminus \sL(H)$ are the $d$ children of $\bu$ which were deleted from $H$, namely $\bu\setminus\{y_i\}\cup\{y\},~i=1,\ldots,d$. All these leaves contain $z\notin R$ hence they do not belong to $P$. In addition, $P\neq \sL(T)$ since $z\notin R$, so we also need to show that $P\neq \sL(H)$. To prove this, we note that there exists $\bv\in\sL(H_1)$ that contains $z$ (we could have chosen a leaf from any of the trees $H_1,\ldots,H_d$ --- but not from $H_{d+1}$ --- and $H_1$ is selected arbitrarily here). Indeed, since the root $\bv_1$ of $H_1$ contains $z$, one of its leaves must also contain it. We claim that $\bv'=\bv\setminus\{z\}\cup\{y\} \in \sL(H)\notin P$. First, $\bv'\in\sL(H)$ since our construction replaces every $z$ with a $y$. Second, $\bv'$ contains a label that appears in $T$ only in vertices of $T_1$. On the other hand, the label $y$ appears in $T$ only in vertices of $T_{d+1}$. Therefore $\bv'$ is not a leaf of $T$, and does not belong to $P$.

We turn to the proof of item~\eqref{lem_item:bos_b}. By the first item, we have that $P\subsetneq \sL(H)$ whence $|P|<|\sL(H)|=dr+1$. Furthermore, if $P=\sL(T)\setminus \sL(T_{\bv})$ then, bu construction, $H$ is obtained from $T$ by removing all the descendants of $\bv$. Therefore, $\bv$ is the only leaf of $H$ that is not in $P$ whence $|P|=dr$. On the other hand, if $|P|=dr$ then $P$ is obtained from $\sL(H)$ by removing one leaf. In consequence, for $j=1,\ldots,d$ , w.l.o.g., $P_j=\sL(H_j)$ whence, by the first item, $P_j=\sL(T_j)$. In particular, $z\in R$ and we have that  
\[
|P_{d+1}|=dr - \sum_{j=1}^{d}(dr_j+1) = 
d\bigg(
\one_{\{z\in R\}}+\sum_{j=1}^{d+1}r_j\bigg) - \sum_{j=1}^{d}(dr_j+1)
=
dr_{d+1}\,.
\] The claim follows by induction.
\end{proof}

Fix the root $\bo=\{1,\ldots,d+1\}$, and consider the families
\[\cL=\left\{\sL(T)~:~\mbox{$T$ is a balanced proper $s$-pedigree for $\bo$}\right\}\]
and
\[
\cP=\left\{P~:~\exists L\in\cL,~\emptyset\neq P\subsetneq L\right\}\,.
\]
Let $z\in [n]$ and $K\subseteq \binom{[n]}{d+1}.$ The link of $z$ in $K$ is
\[
\mathrm{lk}_z(K) = \{\bv\setminus \{z\}~:~z\in\bv\in K\}\,.
\]
\begin{definition}
Let $P\in\cP$. We say that an internal label $z\in \cup_{\bv\in P}\bv\setminus\bo$ is an encircled label of $P$ if $\mathrm{lk}_z(P)=\mathrm{lk}_z(L)$ for every $L\in\cL$ containing $P$. Otherwise, $z$ is called a boundary label of $P$.
\end{definition}

\begin{claim}\label{clm:partition_of_labels}
For every $L\in\cL$ and $P\subseteq L$, every internal label $z\in\cup_{\bv\in L}\bv\setminus \bo$ of $L$ satisfies exactly one of the following: $z$ is an encircled label of $P$, $z$ is an encircled label of $L\setminus P$, or $z$ is a boundary label of both $P$ and $L\setminus P$.
\end{claim}
\begin{proof}
By viewing a set $L\in\cL$ as the $d$-dimensional faces of a stacked triangulation of the simplex $\bo$, and an internal label $z$ as an internal vertex in this triangulation, we observe that the link $\mathrm{lk}_z(L)$ consists of the $(d-1)$-faces of a triangulation of the $(d-1)$-dimensional sphere $\mathbb S^{d-1}$ (and in fact this triangulated sphere is the combinatorial boundary of a stacked polytope).
This observation gives rise to the following intrinsic characterization of an encircled label. Let $P\in\cP$ and $z$ be an internal label that appear in $P$. We claim that $z$ is encircled in $P$ if and only if $\mathrm{lk}_z(P)$ consists of the $(d-1)$-faces of a triangulation of $\mathbb S^{d-1}$. This condition is clearly necessary since the property that the link of $z$ is a triangulation of $\mathbb S^{d-1}$ holds true for every $L\in\cL$. In addition, if for some $L\in\cL$ that contains $P$ there holds $\mathrm{lk}_z(P)\subsetneq\mathrm{lk}_z(L)$ then $\mathrm{lk}_z(P)$ is not a triangulation of $\mathbb S^{d-1}$ --- since no strict subset of $\mathbb S^{d-1}$ is homeomorphic to $\mathbb S^{d-1}$. Hence, this condition is also sufficient. 

We now turn to the proof of the claim. For every internal label $z\in \cup_{\bv\in L}\bv\setminus \bo$ of $L$, its link $\mathrm{lk}_z(P)$ in $P$ is either equal to $\mathrm{lk}_z(L)$, empty, or neither of these. The first case holds true if and only if $\mathrm{lk}_z(P)$ is a triangulation of $\mathbb S^{d-1}$ whence $z$ is an encircled label of $P$. Similarly, the second case is satisfied if and only if $\mathrm{lk}_z(L\setminus P)=\mathrm{lk}_z(L)$ is a triangulation of $\mathbb S^{d-1}$ whence $z$ is an encircled label of $L\setminus P$. In the third case neither $\mathrm{lk}_z(P)$ nor $\mathrm{lk}_z(L\setminus P)$ is a a triangulation of $\mathbb S^{d-1}$ and $z$ is a boundary label of both sets.
\end{proof}

\begin{lemma}\label{lem:partial_enum}
The following holds for every integers $0 \le t \le r \le s$.
\begin{enumerate}[(1)]
    \item The number of sets $P\in\cP$ with $r-t$ encircled labels and $t$ boundary labels is at most
$2(\alpha_dn)^{r}(2^dr)^t$.
\item For every set $P\in\cP$ with $r-t$ encircled labels and $t$ boundary labels, the number of $L\in\cL$ containing $P$ is at most $2(\alpha_dn)^{s-r}(\alpha_d2^d(s-r+t))^t$.
\end{enumerate}
\end{lemma}
\begin{proof}
For the first item, suppose that the set $B\in\binom{\{d+2,\ldots,n\}}{t}$ of the $t$ boundary labels of $P$ is fixed. Recall that, by item~\eqref{lem_item:bos_a} of Lemma~\ref{lem:boston}, every $P\in \cP$ with $r$ internal labels is contained in the set of leaves a proper $r$-pedigree for $\bo$. To construct such a pedigree $H$ we (i) select a $(d+1)$-ary trees with $r$ internal nodes, (ii) split the nodes to the $t$ nodes that will get a boundary label from $B$ and the $r-t$ node that will get an encircled label from $[n]\setminus (B\cup\{1,\ldots,d+1\})$, (iii) select the encircled labels and (iv) assign the labels to the nodes accordingly. We construct $P$ from the pedigree $H$ as following. First, we must include  in $P$ all the leaves containing an encircled label. The remaining leaves are spanned by the $t$ boundary labels and by Lemma~\ref{lem:boston} 
(in which we now let $H$ and $t$ play the roles of $T$ and $r$, respectively) there are at most $dt+\one_{\{r=t\}}$ of them. We conclude the construction of $P$ by selecting a subset of these leaves.

In summary, the number of sets $P\in\cP$ with a specific set $B$ of $t$ boundary labels and $r-t$ encircled labels is at most
\begin{equation}
\sC_d(r)\binom rt\binom{n-d-1-t}{r-t}t!(r-t)!2^{dt+1} \le 
2(\alpha_d n)^{r-t}(\alpha_d2^dr)^t,
\label{eq:P-count-generalized}    
\end{equation}
and the first item is proved by additionally selecting the set $B$.

For the second item, fix a set $P\in \cP$ with the above parameters, and let $B$ denote the set of its $t$ boundary labels. Every $L\in\cL$ containing $P$ is determined by the set $P':=L\setminus P $.  By Claim~\ref{clm:partition_of_labels}, $B$ is the set of boundary labels of $P'$, and in addition $P'$ has $s-r$ encircled labels that do not appear in $P$. Thus, by~\eqref{eq:P-count-generalized}, the number of such sets $P'$ (and in turn, the number of $L\in\cL$ with $L \supseteq P$) is at most
\[ 2(\alpha_d n)^{s-r} (\alpha_d  2^d (s-r+t))^t\,, \]
as claimed.
\end{proof}

\subsection{Proof of Theorem~\ref{thm:supcrit}}
Let $Y \sim \cY_d(n,p)$. Throughout the proof, we denote by $c_d$ a constant depending on $d$ that may change from line to line. We define
\begin{align*}
\mu &= \sum_{L\in\cL}\P\left(\{L\subseteq Y\}\right)\,,\\
\rho&=\sum_{\substack{L_1\ne L_2\in\cL,\\L_1\cap L_2\neq\emptyset}}\P\left(\{L_1,L_2\subseteq Y\}\right)\,.
\end{align*}
The cardinality of $\cL$ is $\sC_d(s')^{d+1}(n-d-1)_s$ and every $L\in\cL$ is of size $ds+1$ whence
\begin{align}
    \mu&=\sC_d(s')^{d+1}(n-d-1)_sp^{ds+1} \nonumber\\
    &=(c_d+o(1))(\alpha_dn)^{-1/d}(s')^{-3(d+1)/2}(1+\epsilon)^{ds+1}=e^{(d(d+1)+o(1))n^{\delta}} \label{eqn:mu}\,.
\end{align}

We can express $\rho$ as
\begin{align*}
\rho&=\sum_{P\in\cP}
\sum_{\substack{L_1\ne L_2\in\cL,\\L_1\cap L_2=P}}\P(\{L_1,L_2\subseteq Y\})
%\\&
=\sum_{P\in\cP}
\sum_{\substack{L_1\ne L_2\in\cL,\\L_1\cap L_2=P}}p^{2(ds+1)-|P|}
\\
&\le \sum_{P\in\cP}|\{L\in\cL~:~P\subseteq L\}|^2\, p^{2(ds+1)-|P|}\,.\\
\end{align*}
The inequality is obtained by summing over all $L_1,L_2$ whose intersection contains $P$ rather than distinct $L_1,L_2$ whose intersection is equal to $P$. By Lemma~\ref{lem:partial_enum} we find
\begin{equation} \label{eqn:rho}
\rho \le 8\sum_{m,r,t} (\alpha_dn)^{m/d-r-2/d}(1+\epsilon)^{2(ds+1)-m}(\alpha_d2^ds)^{3t}\,,
\end{equation}
where the summation of over all admissible tuples $m,r$ and $t$ for which there exists $P\in\cP$ with $|P|=m$, $r-t$ encircled labels and $t$ boundary labels. Combining~\eqref{eqn:rho} and the LHS of~\eqref{eqn:mu} yields
\begin{align}\nonumber
\rho/\mu^2&\le c_d\sum_{m,r,t} (\alpha_dn)^{f/d-r}(1+\epsilon)^{-m}(\alpha_d2^ds)^{3(t+d+1)}\\
\label{eqn:pho_mu2} % I called it pho instead of rho. I miss Saigon Shack too much!!
&\le c_d\max_{m,r,t} (\alpha_dn)^{m/d-r}(1+\epsilon)^{-m}(\alpha_d2^ds)^{3(t+d+2)}\,.
\end{align}
The last inequality follows from the fact that $t\le r\le s$ and $m\le ds$.

We now turn to bound the expression in~\eqref{eqn:pho_mu2}, where the power $m/d-r$ of $n$ plays a key role. Lemma~\ref{lem:boston} asserts that $m/d-r\le  0$, and we claim that if equality holds then $m \ge d(ds'+1)$. Indeed, consider a set $P\in\cP$ for which $m=dr$, and let $T$ be a balanced proper $s$-pedigree satisfying $P\subseteq\sL(T)$. By item~\eqref{lem_item:bos_b} of Lemma~\ref{lem:boston}, the equality $m=dr$ implies that $P=\sL(T)\setminus\sL(T_{\bv})$ for some vertex $\bv$ of $T$. Since $T$ is balanced --- and this is the crucial place we use this assumption --- $T_{\bv}$ contains at most $ds'+1$ leaves, and
\[
m=|P| \ge ds+1 - (ds'+1) = d(ds'+1)\,.
\]

In addition, recall that $P$ is obtained from a proper $r$-pedigree $H$ by removing $dr+1-m$ leaves. Every leaf removal increases the number of the boundary labels of $P$ by  at most $d+1$. Therefore, we find that $t\le (d+1)(dr+1-m)$, and equivalently, $m/d-r \le 1/d-t/(d(d+1))$. 

Denote the expression in~\eqref{eqn:pho_mu2} by 
\[\Xi:= (\alpha_dn)^{m/d-r}(1+\epsilon)^{-m}(\alpha_d2^ds)^{3(t+d+2)}\,,\] 
and consider the following three cases.

\begin{enumerate}[\bf {Case} (i).]
    \item $m=dr$. By the discussion above, this can occur only if $t\le d+1$ and $m\ge d(ds'+1)$. Therefore,
    \[
    \Xi
    \le e^{-(d^2+o(1))n^{\delta}}\,.
    \]
    \item $t \ge 2(d+1)$. We recall that $m/d-r\le 1/d-t/(d(d+1))$ and find
    \begin{equation}
        \label{eq:case2}
    \Xi
    \le
     (\alpha_dn)^{1/d-t/(d(d+1))}(\alpha_d2^ds)^{3(t+d+2)}\,.
\end{equation}
    The derivative of the logarithm of the RHS of~\eqref{eq:case2} with respect to $t$ is
    \[
    -\frac{\log{n}}{d(d+1)}+3\log{s} + O_d(1) = \log{n}\left(-\frac{1}{d(d+1)}+6\delta\right) +O_d(1)\,,
    \]
    which is negative for all sufficiently large integers $n$. Therefore, we can bound $\Xi$ by assigning $t=2(d+1)$ in the RHS of~\eqref{eq:case2}, yielding 
    \[
    \Xi\le n^{-\frac{2d+3}
    {3d(d+1)^2}+o(1)}\,.
    \]
    \item $m\le dr-1$ and $t\le 2d+1$. By a direct assignment we have that 
        \begin{align*}
        \Xi\le n^{-\frac{1}
    {d(d+1)}+o(1)}\,.
    \end{align*}
\end{enumerate}
In summary, and using~\eqref{eqn:pho_mu2}, we find that
\begin{equation}\label{eq:rho/mu2-bound}
\rho/\mu^2 \le 
n^{-\frac{2d+3}{3d(d+1)^2}+o(1)}\,.
\end{equation}
Substituting~\eqref{eqn:mu} and~\eqref{eq:rho/mu2-bound} in
Janson's inequality then establishes that 
\[
\P\bigg(\bigcap_{L\in\cL}\{L\not\subseteq Y\}\bigg) \le
\exp\Big( - \frac{\mu^2}{\mu+\rho} \Big)
%\exp\Big(\frac{-1}{1/\mu+\rho/\mu^2}\Big)
= \exp\Big(n^{\frac{2d+3}{3d(d+1)^2}+o(1)} \Big) 
< \exp\Big(-n^{12\delta}\Big)\,,
\]
with the last inequality valid for large enough $n$. This completes the proof.
\qed

\section{A closer look at the subcritical regime}\label{sec:subcritical_zoom}
In this section we prove Theorem~\ref{thm:density}. Let $0\le \gamma <\alpha_d^{-1/d}$, $p=\gamma n^{-1/d}$ and $Y\sim\cY_d(n,p)$.
Recall from ~\S\ref{subsection:stacked} that the smallest positive root $\hat\gamma $ of $Q_\gamma(x)=x^{d+1}-x+\gamma$ is given by
\[
\hat\gamma := \sum_{s=0}^{\infty}\sC_d(s)\gamma^{ds+1}\,.
\]
For an integer $s\ge 0$, we denote by $\mathfrak T_{s}$ the set of $(d+1)$-ary trees with $s$ internal vertices, and let
$\mathfrak T_{\le s}=\cup_{s'\le s}\mathfrak T_{s'}$. In addition, for every $s\ge 0$ and a tree $T\in\mathfrak T_s$, let $X_T$ be the number of proper $s$-pedigrees whose leaves are all members of $Y$ and whose underlying tree is $T$.
We prove Theorem~\ref{thm:density} using the following lemmas.
\begin{lemma}\label{lem:X_T}
Fix an integer $s\ge 0 $ and $T\in\mathfrak T_s$.
Then,
\[
n^{1/d}\frac{X_T}{\binom n{d+1}} \to \gamma^{ds+1}
\]
in probability as $n\to\infty$.
\end{lemma}
\begin{proof}
We enumerate the proper $s$-pedigrees whose underlying tree is $T$ by its root $\bo\in\binom{[n]}{d+1}$ and the internal labels $z_1,\ldots,z_s$ given in some predetermined order (e.g., by the order of their appearance in a Breadth-First-Search on $T$).% Note that this enumeration is equivalent to enumerating the pedigrees by their set of leaves, since the root and the sequence of internal labels of a proper $s$-pedigree can be decoded from its set of leaves. Indeed, we find the root by locating the $(d-1)$-faces that appear as a subset of only one leaf (these are the facets of the boundary of the simplex --- whose vertices form the root). We then apply the following observation to decode the BFS order of the internal labels: the outgoing label of an internal vertex $\bv$ in a proper pedigree is the only label that appear in a leaf with every label $z'\in\bv$.

The expectation of $X_T$ is given by
\begin{align}
\E[X_T] &= \binom n{d+1}(n-d-1)_s(\gamma n^{-1/d})^{ds+1} \label{eq:Expectation_XT}
\\
&= (1+o(1))\binom n{d+1}n^{-1/d}\gamma^{ds+1}\,,    \nonumber
\end{align}
where the transition to the second line follows from $(n-d-1)_sn^{-s}\to 1$.
 
We proceed to compute the variance of $X_T$. Let $T_1,T_2$ be proper $s$-pedigrees whose underlying tree is $T$, let $q$ denote the total number  of common labels (internal and root) in $T_1$ and $T_2$, and write $P:=\sL(T_1)\cap \sL(T_2)$ and $t=|P|$. 

Clearly, $q=0$ implies $t=0$. In addition, 
we claim that if $0<q\le s+d+1$ then $$t\le dq-1\,.$$ Indeed, we first observe that if $t=ds+1$ then $T_1$ and $T_2$ have the same label set, thereby $q=s+d+1$ and the inequality is satisfied. Otherwise, $P\subsetneq \sL(T_1).$ Let $R$ denote the internal labels of $T_1$ that are contained in one of the sets in $P$. Note that the labels in $R$ must also appear in $T_2$ --- either as internal labels or in the root. Furthermore, by Lemma~\ref{lem:boston}, $|P|\le d|R|$. If this inequality is strict, then the claim follows from $|R|\le q$. Otherwise, Lemma~\ref{lem:boston}  asserts that $P=\sL(T_1)\setminus\sL((T_1)_{\bv})$ for some $\bv\in T_1$. In such a case, besides the $t/d$ common labels of $R$, all the $d+1$ labels in the root of $T_1$ also appear in one of the sets of $P$ --- and therefore also in $T_2$. Consequently, an even stronger inequality of $q \ge t/d + d+1$ is obtained.

Therefore, by summing over the integers $0\le q\le d+s+1$, we find
\begin{align*}
\E[X_T^2] \le~ \E[X_T]^2+\sum_{q=1}^{s+d+1}Cn^{2s+2d+2-q}(\gamma n^{-1/d})^{2ds+2-(dq-1)}\,,
%\E[X_T]^2 + C\,\sum_{t=1}^{ds} n^{2d+2s+2-\frac{t+1}d}p^{2ds+2-t}+\E[X_T],
\end{align*}
where the constant $C=C_{s,d}$ accounts for assigning the selected $2s+2d+2-q$ labels to the nodes of $T$ that creates $T_1,T_2$. Therefore $\var(X_T)=O(n^{2d+2-3/d}),$ and the lemma is derived via Chebyshev's inequality.
\end{proof}

\begin{lemma}\label{lem:pairs}
Let $s_0\ge 0 $ be an integer, and let $Z_{s_0}$ denote the number of pairs of proper pedigrees with at most $s_0$ internal vertices that have the same root and whose leaves are contained in $Y$. Then,
\[
\mathbb E[Z_{s_0}] \le Cn^{d+1-2/d}\,,
\]
where $C$ is a constant depending on $d,s_0$.
\end{lemma}

\begin{proof}
Let $T_1,T_2$ be distinct proper pedigrees sharing the same root $\bo$. Denote by $R_i$ the set of internal labels of $T_i$, $i=1,2,$ and let $s=|R_1\cup R_2|$. We claim that
\begin{equation}\label{eq:RL_claim}
|\sL(T_1)\cup\sL(T_2)| \ge ds+2\,.
\end{equation}
Indeed, since $|\sL(T_i)|=d|R_i|+1$, \eqref{eq:RL_claim} reformulates to
\[
|\sL(T_1)\cap\sL(T_2)| \le d|R_1\cap R_2|\,,
\]
which is derived by Lemma~\ref{lem:boston}.
We conclude the lemma by the following crude bound.
\begin{align*}
    \E[Z_{s_0}] &\le \binom{n}{d+1} |\mathfrak T_{\le s_0}|^2 \sum_{s=1}^{2s_0} \binom{n-d-1}{s} s^{2s_0} (\gamma n^{-1/d})^{ds+2}
    \\
    &\le C\, n^{d+1-2/d}\,.
\end{align*}
Here we chose a pair $T_1,T_2$ be selecting the common root, the two underlying $(d+1)$-ary trees and the label set $R_1\cup R_2$, and afterwards assigning the outgoing labels to the (at most $2s_0$) internal vertices. Finally, we use \eqref{eq:RL_claim} to get a lower bound for the number of leaves that are required to be in $Y$.
\end{proof}
\begin{proof}[Proof of Theorem~\ref{thm:density}]
Let $\epsilon>0$, and consider a sufficiently large integer $s_0$ satisfying
\[
\sum_{s=0}^{s_0}\sC_d(s)\gamma^{ds+1}>\hat\gamma-\epsilon\,.
\]

We start with the lower bound. Note that
\begin{equation}
\label{eq:XTZ}
\binom{n}{d+1}\hat X \ge \sum_{T\in\mathfrak T_{\le s_0}}X_T -  Z_{s_0}\,.    
\end{equation}
Indeed, the summation on the RHS counts the number of special pedigrees (i.e., proper pedigrees with at most $s_0$ internal vertices) whose leaves are all members of $Y$, from which we subtract the number of pairs of such pedigrees with the same root. Therefore, every root of such a pedigree is accounted for at most once.

Since $|\mathfrak T_{\le s_0}|$ does not depend on $n$, by Lemma~\ref{lem:X_T} we have that 
\begin{align}\label{eq:limXT}
n^{1/d}\frac{1}{\binom n{d+1}} \sum_{T\in\mathfrak T_{\le s_0}}X_T
= 
\sum_{s=0}^{s_0}\sum_{T\in\mathfrak T_{s}}n^{1/d}\frac{X_T}{\binom n{d+1}} 
\to& \sum_{s=0}^{s_0}\sum_{T\in\mathfrak T_{s}}\gamma^{ds+1}\\\nonumber =&
\sum_{s=0}^{s_0}\sC_d(s)\gamma^{ds+1} >\hat\gamma-\epsilon\,,
\end{align}
where the convergence is in probability as $n\to\infty.$
In addition, by Lemma~\ref{lem:pairs}, $\E[n^{1/d}Z_{s_0}/\binom{n}{d+1}] \to 0$ as $n\to \infty$, and we conclude the lower bound using \eqref{eq:XTZ}.

To establish a matching upper bound, we bound the number of sets $\bo\in\binom{[n]}{d+1}$ that have a stacked contraction in $Y$ by
\begin{equation}
    \label{eq:hatX_upper}
    {\binom n{d+1}}\hat X \le\sum_{T\in\mathfrak T_{
    n-d-1}}X_T +W\,,
\end{equation}
where $W$ is the number of {\em improper} $(m,l,s)$-pedigrees (i.e., $\aG:=l-ds-1\ge 1$) whose leaves are all members in $Y$. 

We split the first summation by considering big and small trees in separate. Namely, using \eqref{eq:Expectation_XT} we find that
\begin{align*}
\E\bigg[\sum_{s=s_0+1}^{n-d-1}\sum_{T\in\mathfrak T_s}X_T\bigg] &< \binom{n}{d+1}n^{-1/d} \sum_{s>s_0}\sC_d(s)\gamma^{ds+1}
\\
&< \binom{n}{d+1}n^{-1/d}\epsilon\,,
\end{align*}
whence
\begin{equation}
\mathbb P\bigg(n^{1/d}\frac{1}{\binom n{d+1}}\sum_{s=s_0+1}^{n-d-1}\sum_{T\in\mathfrak T_s}X_T > \sqrt{\epsilon}\bigg) \le \sqrt{\epsilon}\,.
\label{eq:W2}
\end{equation}

In addition, since $\sum_{s\le s_0}\sC_d(s)\gamma^{ds+1}<\hat\gamma,$ we find by \eqref{eq:limXT} that 
\begin{equation}
\mathbb P\bigg(n^{1/d}\frac{1}{\binom n{d+1}}\sum_{s\le s_0}\sum_{T\in\mathfrak T_s}X_T > \hat\gamma\bigg) \to 0\,,
    \label{eq:W3}
\end{equation} 
as $n\to\infty$.

To bound the contribution of $W$, we recall several facts from Section~\ref{sec:subcritical}. Let $\mathcal B$ denote the event that there is no $(m,l,s)$-pedigree with $s>(\log n)^{5/2}$ distinct internal labels whose leaves are all members of $Y$. By Observation~\ref{obs:reduce-labels} and \eqref{eq:cE_s^f}, $\mathcal B$ occurs with probability $1-o(1)$. 

In addition, recall that (as stated in the last paragraph of the proof of Theorem~\ref{thm:subcrit-main}) for every fixed $\bo\in\binom{[n]}{d+1}$, $s\le (\log n)^{5/2}$ and $\aG \ge 1$, the probability that there exists an $(l,m,s)$-pedigree, such that $l=ds+1+\aG$, whose root is $\bo$ and whose leaves are all members in $Y$ is at most $e^{-\delta s}n^{-\frac{\aG+1}d+o(1)}$, where $\delta = 1 - \gamma\alpha_d^{1/d}.$ Therefore, by summing over $\bo,s$ and $\aG\ge 1$ we find that
\[
\E[W \, \one_\mathcal B] \le \binom{n}{d+1}n^{-2/d+o(1)}\,,
\]
whence
\begin{equation}
\mathbb P\bigg(n^{1/d}\frac{1}{\binom{n}{d+1}} W > \sqrt{\epsilon}\bigg) \le n^{-1/d+o(1)}/\sqrt{\epsilon} + (1-\mathbb P(\mathcal B)) \to 0\,,
\label{eq:W1}
\end{equation} 
as $n\to\infty$.
In summary, and using \eqref{eq:hatX_upper},\eqref{eq:W2},\eqref{eq:W3} and \eqref{eq:W1}, we find that as $n\to\infty$,
\[
\mathbb P(n^{1/d}\hat X > \hat\gamma+2\sqrt\epsilon) \le \sqrt{\epsilon}+o(1)\,,
\]
as claimed.
\end{proof}

\appendix

\section{The leaf-to-label differential in the DAGs}

\begin{proposition}\label{prop:a-geq-0}
Let $d\geq 2$, and suppose $T$ is a stacked contraction with $l$ faces and a total of $s+d+1$ distinct labels. Then 
\begin{equation}\label{eq:ell-s-bound}
l \geq d s + 1\,.
\end{equation}
Moreover, there is an explicit $A \in \R^{d(s+d+1) \times l}$ such that $\dim(\ker A) = l-(ds+1)$.
\end{proposition}
\begin{proof}
Suppose without loss of generality that the label set is $\{1,\ldots,s+d+1\}$, and let $v_1,\ldots,v_{s+d+1}\in \R^d $ be in general position. Let $x_1<x_2<\ldots<x_{d+1}$; we represent a face $f=
\{x_1,\ldots,x_{d+1}\}$ via a vector $w_f \in \R^{d(s+d+1)}$
as follows. Define the matrix 
\begin{equation}
    \label{eq:Af-def}
    A_f := \left(\begin{array}{c|c|c}
    v_{x_1} & 
    \vphantom{\bigg(}\ldots &
    v_{x_{d+1}} \\
    \mbox{\small$1$}  & & \mbox{\small$1$}
    \end{array}\right)\,,
\end{equation}
whose determinant is --- up to the sign --- the volume of the simplex whose vertices are $v_{x_1},\ldots,v_{x_{d+1}}$. The cofactor $C_{i,j}(A_f)$ of the matrix $A_f$, for every $1\le j\le d+1$ and $1\le i\le d$, is the partial derivative of $\det(A_f)$ with respect to the $i$-th coordinate of $v_{x_j}$.
Representing the vector $w_f\in \R^{d(s+d+1)} $ as an $(s+d+1)\times d$ matrix, we define $w_f$ to have
\begin{equation}\label{eq:w_f-def}
(w_f)_{x_j,i} := C_{i,j}(A_f) \qquad{1\leq j\leq d+1,1\leq i\le d}\,,
\end{equation}
and $0$ elsewhere. To a set $K\subset\binom{[s+d+1]}{d+1}$ we associate the matrix $A_K$ whose columns are $w_f,~f\in K$.

Let $z\in\R^{d(s+d+1)}$ be a vector that is also represented as an  $(s+d+1)\times d$ matrix whose columns are $z_1,\ldots,z_{s+d+1}$, and consider a continuous motion of the vertices
\[
v_j^{(t)}= v_j+t\cdot z_j,~j=1,\ldots,s+d+1.
\]
Let $A_f^{(t)}$ be defined as~\eqref{eq:Af-def} with respect to $v_{x_j}^{(t)},~j=1,\ldots,d+1$. Then, the derivative of the determinant of $A_f^{(t)}$ at time $t=0$ equals the inner product of $z$ and $w_f$. In words, $z$ is orthogonal to $w_f$ if and only if the above motion infinitesimally preserves the volume of the simplex whose vertices start at $v_{x_1},\ldots,v_{x_{d+1}}$. Clearly, any translation of $\R^d$, where $z_j=u$ for all $j=1,\ldots,s+d+1$, preserves the volume of every simplex. Additionally, we claim that for every matrix $M\in\R^{d\times d}$ having $\mathrm{tr}(M)=0$, setting $z_j=Mv_j,~j=1,\ldots,s+d+1$ infinitesimally preserves the volume of every simplex $f$. Indeed, In such a case $v_j^{(t)}=(I+tM)v_j$ whence the volume of the simplex $f$ at time $t$ satisfies 
\[
\det(A_f^{(t)})=\det(I+tM)\det(A_f).
\]
We conclude the claim by observing that the derivative of $\det(I+tM)$ equals $\mathrm{tr}(M)=0$. In summary, we showed that the left kernel of $A_K$ is of dimension at least $d+d^2-1$, so its rank is at most $d(s+d+1)-(d+d^2-1)=ds+1$. In the remainder of the appendix, we show that this bound is attained by the faces of every stacked contraction.

\begin{claim}\label{clm:triangulation-linear-comb}
Let $x_1<\ldots<x_{d+1}$ and $z$ be such that $ x_{k-1} < z < x_{k}$ for some $1\leq k \leq d+2$. If $f=(x_1,\ldots,x_{d+1})$ and $f_i$ $(i=1,\ldots,d+1$) is the face obtained from it by replacing $x_i$ with $z$, then for every $v_{x_1},\ldots,v_{x_{d+1}},v_z \in \R^d$,
\begin{equation}
    \label{eq:w_f-triangulation}
    w_{f} = \sum_{r=1}^{k-1} (-1)^{k+r+1} w_{f_r}
    + \sum_{r=k}^{d+1} (-1)^{k+r} w_{f_r}
\end{equation}
\end{claim}
\begin{proof}
Let $1\leq i,j\leq d+1$, and consider the matrix 
\[ E_i = 
\left(\begin{array}{c|c|c|c|c|c|c}
    \mbox{\small$1$} & \mbox{\small$1$} &  \mbox{\small$1$} &  \mbox{\small$1$}&  \mbox{\small$1$}& \mbox{\small$1$} &  \mbox{\small$1$}\\
    v_{z}^{(\neg j)} & 
    v_{x_1}^{(\neg j)} & 
    \vphantom{\Bigg(}\ldots &
    v_{x_{i-1}}^{(\neg j)}&
    v_{x_{i+1}}^{(\neg j)}& \ldots & 
    v_{x_{d+1}}^{(\neg j)} \\
    \mbox{\small$1$} & \mbox{\small$1$} &  \mbox{\small$1$} &  \mbox{\small$1$}&  \mbox{\small$1$}& \mbox{\small$1$} &  \mbox{\small$1$}
    \end{array}
    \right)\,,
\]
where $v^{(\neg j)}$ represents the vector obtained from $v$ by omitting its $j$-th coordinate.
Observe that the submatrix of $E_i$ obtained by omitting the first row and the column containing $v_z$ is precisely the submatrix obtained by omitting from $A_f$ the $j$-th row and the column containing $v_{x_i}$. Furthermore, omitting the first row and the column of $v_{x_r}$ from $E_i$ (for $1\leq r \leq d+1$, $r\neq i$) gives the submatrix of $A_{f_r}$ obtained by omitting its $j$-th row and the column of $v_{x_i}$, followed by moving the column of $v_z$ to be the first one (retaining the inner ordering between the remaining columns).
Observe that the order of the columns in $A_{f_r}$ is such that:
\begin{enumerate}[(a)]
    \item the vector $v_{x_i}$ appears in column $i-\one_{\{r < i\}} + \one_{\{k\leq i\}}$;
    \item if we omit the column of $v_{x_i}$, then $v_z$ appears in column $k-\one_{\{r < k\}} - \one_{\{i<k\}}$.
\end{enumerate}
It then follows from the observations above and~\eqref{eq:w_f-def} that, if $i\geq k$, then
\begin{align*}
\sum_{r=1}^{k-1} (-1)^{r} C_{1,r+1}(E_i) &= \sum_{r < k} (-1)^{r+i+j} (w_{f_r})_{x_i,j}
\end{align*}
(for each $r$ the corresponding cofactor of $A_{f_r}$ contributes $(-1)^{i+j}$), and
\begin{align*}
\sum_{r=k}^{d} (-1)^{r+1} C_{1,r+1}(E_i) &= \sum_{r=k}^{i-1}(-1)^{r+i+j} (w_{f_r})_{x_i,j}+ \sum_{r=i}^{d}(-1)^{r+(i+1)+j} (w_{f_{r+1}})_{x_i,j}\\
&= \sum_{\substack{r \geq k\\ r\neq i}} (-1)^{r+i+j}(w_{f_r})_{x_i,j}\,.
\end{align*}
Combining these, along with the fact that 
 $(w_{f_i})_{x_i,j}=0$ for all $j$ (since $f_i$ does not contain the label $x_i$, having replaced it by $z$), we find that for $i\geq k$,
 \begin{align}
     \label{eq:cofactor-i-geq-k}
     \sum_{r=1}^d (-1)^r C_{1,r+1} (E_i) = \sum_{r< k} (-1)^{r+i+j} (w_{f_r})_{x_i,j}
     - \sum_{r\geq k} (-1)^{r+i+j} (w_{f_r})_{x_i,j}\,.
 \end{align}
 
 On the other hand, if $i<k$ then
 \begin{align*}
\sum_{r=k-1}^{d} (-1)^{r} C_{1,r+1}(E_i) &= \sum_{r=k}^{d+1}(-1)^{(r+1)+i+j} (w_{f_r})_{x_i,j}\,,
 \end{align*}
whereas
\begin{align*}
\sum_{r=1}^{k-2} (-1)^{r+1} C_{1,r+1}(E_i) &= \sum_{r=1}^{i-1} (-1)^{r+(i-1)+j} (w_{f_r})_{x_i,j} + \sum_{r=i}^{k-2}(-1)^{r+i+j} (w_{f_{r+1}})_{x_i,j} \\
&= \sum_{r=1}^{k-1} (-1)^{r+i+j+1} (w_{f_r})_{x_i,j}
\,,
\end{align*}
and we see that~\eqref{eq:cofactor-i-geq-k} also holds for $i<k$ (again using that $(w_{f_i})_{x_i,j}=0$ for all $j$).

Recalling that $C_{1,1}(E_i) = (-1)^{i+j}(w_f)_{x_i,j}$, 
we may now, for every $i=1,\ldots,d+1$, substitute~\eqref{eq:cofactor-i-geq-k} in the expansion $\sum_r C_{1,r}(E_i) = \det(E_i)=0$, and obtain that
\[
    (w_f)_{x_i,j} = \sum_{r<k} (-1)^{k+r+1} (w_{f_r})_{x_i,j}
    + \sum_{r\geq k} (-1)^{k+r} (w_{f_r})_{x_i,j}\,.
\]

It remains to verify~\eqref{eq:w_f-triangulation} at the coordinate $z,j$. To this end, note first that $(w_f)_{z,j}=0$ for all $j$ (as the face $f$ does not contain the label $z$), and let
 \[ E = 
\left(\begin{array}{c|c|c}
    \mbox{\small$1$} & & \mbox{\small$1$}\\
    v_{x_1}^{(\neg j)} & 
    \vphantom{\Bigg(}\ldots &
    v_{x_{d+1}}^{(\neg j)} \\
    \mbox{\small$1$} & & \mbox{\small$1$}
    \end{array}
    \right)\,.
\]
For every $r=1,\ldots,d+1$, the $(1,r)$-minor of $E$ is exactly the minor corresponding to omitting the $j$-th row and the column containing $v_z$ from $A_{f_r}$. Therefore,
\begin{itemize}
    \item for $r=1,\ldots,k-1$ (the vector $v_z$ appears in column $k-1$ of $A_{f_r}$) we have
    \[ C_{1,r}(E) = (-1)^{r+k+j} (w_{f_r})_{z,j}\,;\]
    \item for $r=k,\ldots,d+1$ (the vector $v_z$ appears in column $k-1$ of $A_{f_r}$) we have
    \[ C_{1,r}(E) = (-1)^{r+1+k+j} (w_{f_r})_{z,j}\,.\]
\end{itemize}
Altogether, recalling definition~\eqref{eq:w_f-def} and that $\sum_r C_{1,r}(E) = \det(E)=0$ shows that 
\[ \sum_{r<k} (-1)^{k+r} (w_{f_r})_{z,j} +
\sum_{r\geq k} (-1)^{k+r+1} (w_{f_r})_{z,j} = 0\,,\] 
at which point the fact that $(w_f)_{z,j}=0$ for all $j$ establishes~\eqref{eq:w_f-triangulation} at coordinate $(z,j)$ and thereby concludes the proof.
\end{proof}

\begin{claim}\label{clm:new-z-indep}
Let $x_1<\ldots<x_{d+1}$ and $z$ be such that $z\neq x_i$ for all $i$. Let $f_i$ ($i=1,\ldots,d+1$) be the face with labels $\{x_1,\ldots,x_{d+1},z\}\setminus\{x_i\}$. If $\mathcal F$ is any collections of faces that do not contain the label $z$, and $\tilde w$ is any non-zero vector in the span of $\{w_{f}:  f\in \mathcal F\}$, then $  \tilde w ,w_{f_1},\ldots,w_{f_d}$ are linearly independent.
\end{claim}
\begin{proof}
Denote by $1\leq k \leq d+2$ the integer such that $x_{k-1}<z<x_s$, and suppose that for some $a_0,\ldots,a_d \in\R$ we have
\[ a_0 \tilde w + \sum_{i=1}^{d} a_i w_{f_i} = 0\,.\]
Since $\tilde w$ is a combination of $\{w_{f} : f\in\mathcal F\}$ where $\mathcal F$ does not feature the label $z$, every $f\in\mathcal F$ has $(w_{f})_{z,j}=0$ for all $j$, from which we infer that
\begin{align}
    \label{eq:linear-comb-at-z}
    \sum_{i=1}^d a_i (w_{f_i})_{z,j} = 0 \qquad\mbox{for all $j=1,\ldots,d$}\,.
\end{align} 
Define
\begin{align}
    \label{eq:E-v-def}
    E = \left(\begin{array}{cccc}
    & v_{x_1}^\textsc{t} && \mbox{\small$1$}\\
    \noalign{\smallskip}
    \hline
    & \vdots && \vdots \\
    \hline
    \noalign{\smallskip}
    & v_{x_{d+1}}^\textsc{t} && \mbox{\small$1$}
    \end{array}
    \right)
    \qquad\mbox{and}\qquad
    \hat v = \left(\begin{array}{c}\hat a_1\\ \vdots \\ \hat a_d \\ \mbox{\small$0$}\end{array}\right)\,,
\end{align}
where
\begin{align}
    \label{eq:a_i^*-def}
     \hat a_i := (-1)^{i+k+\one_{\{i<k\}}} a_i \quad(i=1,\ldots,d)\,, 
\end{align}
and let $E_j$ be the matrix obtained from $E$ by replacing its $j$-th column by $\hat v$.

With~\eqref{eq:Af-def} in mind, observe that for every $1\leq i,j \leq d$, omitting the $i$-th row and $j$-th column of $E_j$ gives the transpose of the submatrix obtained by omitting the column of $v_z$ and the $j$-th row in $A_{f_i}$. Thus, by~\eqref{eq:w_f-def} and the fact that $v_z$ appears in column $k-\one_{\{i<k\}}$ in $A_{f_i}$, the $(i,j)$-minor of $E_j$ is  $(-1)^{j+k-\one_{\{i<k\}}} (w_{f_i})_{z,j}$. We conclude that for every $j=1,\ldots,d$,
\[ \det(E_j) = \sum_{i=1}^d \hat a_i (-1)^{i+k-\one_{\{i<k\}}}
(w_{f_i})_{z,j} = \sum_{i=1}^d  a_i
(w_{f_i})_{z,j} = 0\,,
\]
with the last equality due to~\eqref{eq:linear-comb-at-z}. At the same time, we have that $u=(u_j)_{j=1}^{d+1}$ given by
\[ u_j := \frac{\det(E_j)}{\det(E)}\] is the unique solution of $E u = \hat v$ by Cramer's rule, relying here on the hypothesis that $v_1,\ldots,v_k$ are in general position in $\R^d$, whence $\det(E)\neq 0$. We have shown above that $u_j=0$ for all $1\leq j\leq d$, and so $u = u_{d+1} e_{d+1}$, where $e_{d+1}$ is the $(d+1)$-th standard basis vector in $\R^d$. But $E e_{d+1} = \underline \one$, the all-1 vector, whence 
\[ \hat v = E u = u_{d+1} \underline\one\,.\]
From this and the fact that $\hat v_{d+1}=0$ we can then deduce that $\hat v = 0$, and in turn, that $a_i=0$ for all $i=1,\ldots,d$. This reduces the original hypothesis to having $a_0 \tilde w = 0$, whence $a_0=0$ as well, and the proof is concluded.
\end{proof}

We claim that if the oriented faces $g_1,\ldots,g_l$ of the triangulation $T$ form a stacked triangulation of $f_0$ then the $d (s+d+1)\times l$ matrix $A$ whose columns are $\{ w_{g_i}\}_{i=1,\ldots,l}$ has rank at least $ds+1$, which will imply~\eqref{eq:ell-s-bound}.

 This will be proved by induction on the number of internal vertices $m$ in the stacked triangulation $T$. The base case of $m=0$ internal vertices (where $T$ is the single face $f_0$) corresponds to $s=0$ and $l=1=ds+1$.

Every stacked triangulation with $m\geq 1$ faces is obtained by subdividing some face $f$ of a stacked triangulation $\tilde T$ with $m-1$ internal vertices. If $f=(x_1,\ldots,x_{d+1})$, this corresponds to replacing the face $f$ by the $d+1$ faces $f_i$ ($i=1,\ldots,d+1$) as in where $f_i$ has $x_i$ swapped by $z$.

Claim~\ref{clm:triangulation-linear-comb} immediately implies that $\rank(A)\geq \rank(\tilde A)$, where $\tilde A$ is the matrix corresponding to $\tilde T$, the triangulation before subdividing $f$ via the label $z$. In particular, if the label $z$ has already appeared in $\tilde T$, then $\rank(A) \geq \rank(\tilde A) \geq ds+1$ holds true by our induction. 

On the other hand, if $z$ is a new label, then $\rank(A) \geq \rank(\tilde A) + d$ by virtue of Claim~\ref{clm:new-z-indep}.
\end{proof}
    
\subsection*{Acknowledgment} 
The authors thank Eran Nevo for useful discussions.
E.L.~was supported in part by NSF grants DMS-1812095 and DMS-2054833.

\bibliographystyle{abbrv}
\bibliography{triangs}

\end{document}